\documentclass[11pt,A4]{article}
\usepackage{amsmath,amsthm,amsfonts,amssymb,amscd, amsxtra,color}
\usepackage[active]{srcltx}
\usepackage{url}
\usepackage{cite}
\usepackage{multirow}
\usepackage[margin=2.5cm,nohead]{geometry}
\newtheorem{theorem}{Theorem}
\newtheorem{lemma}{Lemma}
\newtheorem{corollary}{Corollary}
\newtheorem{proposition}{Proposition}
\newtheorem{remark}{Remark}

\newtheorem{example}{Example}

\def\min{\operatorname{Minimize}}

\newcommand{\lng}{\langle}
\newcommand{\rng}{\rangle}

\newcommand{\R}{\mathbb R}

\newcommand{\T} { \rm T}

\newcommand{\ds}{\displaystyle}

\begin{document}

\title{A semi-smooth Newton method for projection equations and linear complementarity problems with respect to the second order cone\thanks{This work was supported by CNPq (Grants 303492/2013-9, 474160/2013-0, 305158/2014-7) and FAPEG.}}

\author{
J.Y. Bello Cruz\thanks{Department of Mathematical Sciences, Northern Illinois University, WH 366, DeKalb, IL - 60115, USA (E-mail: {\tt yunierbello@niu.edu}).}
\and
 O. P. Ferreira\thanks{IME/UFG, Avenida Esperan\c{c}a, s/n, Campus Samambaia,  Goi\^ania, GO - 74690-900, Brazil (E-mails: {\tt orizon@ufg.br}, {\tt lfprudente@ufg.br}).}
\and
S. Z. N\'emeth \thanks{School of Mathematics, The University of Birmingham, The Watson Building, Edgbaston, Birmingham - B15 2TT, United Kingdom
(E-mail: {\tt nemeths@for.mat.bham.ac.uk}).} 
\and
L. F. Prudente \footnotemark[3]
}

\maketitle

\begin{abstract}
\noindent In this paper a special semi-smooth equation  associated to the second order cone is studied. It is shown that,  under mild
assumptions,  the  semi-smooth Newton method applied to this equation  is  well-defined and  the generated   sequence  is  globally and Q-linearly
convergent  to a solution.  As an application,  the obtained results  are used  to  study  the  linear second order cone complementarity
problem, with special emphasis on the particular case of positive definite matrices. Moreover, some computational experiments designed to
investigate the practical viability of the method are presented. 

\medskip



\end{abstract}

%%%%%%%%%%%%%%%%%%%%%%%%%%%%%%%%%%%%%%%%%%%%%%%%%%%%%%%%%%%%%%
\section{Introduction}
%%%%%%%%%%%%%%%%%%%%%%%%%%%%%%%%%%%%%%%%%%%%%%%%%%%%%%%%%%%%%%
 In this paper we consider  the following   special semi-smooth equation in $x\in\R^n$ associated to  the closed and convex cone ${\cal{K}}\subseteq\R^n$:
\begin{equation}\label{eq:pwls}
		{\rm P}_{\cal{K}}({ {\rm x}}) +{\T} { {\rm x}}=b, 
\end{equation} 
where   \( b \in \mathbb{R}^n\)  is a constant vector, \(\T\)  is an  \(n\times n \) constant nonsingular  real matrix and  \({\rm P}_{\cal{K}}({\rm x})\) 
denotes the Euclidean metric projection of a vector \({\rm x}\) onto the cone ${\cal{K}}$.  The equation \eqref{eq:pwls} associated to the
positive orthant, ${\cal{K}}=\mathbb{R}^n_{++}$, was first studied    in  \cite{BrugnanoCasulli2007}.  Additional  papers  dealing with  \eqref{eq:pwls}  and its variations had appeared, for instance, in    \cite{BFN2015, BBCFN2016, BFP2016,   BrugnanoCasulli2009,  CasulliZanolli2012, ChenRavi2010, FerreiraNemeth2014,  GriewankBerntRadonsStreubel2015,  HuHuangQiong11, Mangasarian2009, MangasarianMeyer2006, SunLeiLiu2015}.  

The  purpose  of  the present   paper  is   to discuss  the   semi-smooth Newton method to solve equation \eqref{eq:pwls} associated to the  {\it second order cone}
\begin{equation} \label{eq:lc}
{\cal{K}}:=\left\{{\rm x}:=(x_1, {{\rm x}_2})  \in \mathbb{R}\times {\mathbb{R}}^{n-1} ~:~ \|{{\rm x}_2 }\|\leq x_1    \right\}.
\end{equation}
It is shown that,  under mild assumptions,  the  semi-smooth Newton method applied to this equation  
is  well-defined and  the generated   sequence  is  globally and Q-linearly convergent  to a solution. 
 As an application, we use the obtained results  to  study  the {\it linear second order cone complementarity problem (LSOCCP)} : Find ${\rm x}\in \mathbb{R}^n $ such that
\begin{equation} \label{eq:lclcp}
{\rm x}\in {\cal{K}}, \qquad  {\rm M}{\rm x}+{\rm q}\in  {\cal{K}},  \qquad \left\langle {\rm M}{\rm x}+{\rm q}, \, {\rm x}\right\rangle =0, 
\end{equation}
where \( {\rm q} \in \mathbb{R}^n\) is a constant vector, \( {\rm M}\)  is an  \(n\times n \) constant  nonsingular  real matrix. Complementarity
problems related to the second order cone are considered in \cite{GajardoSeeger14, KongXiuHan08,  MalikMohan05}.   This
topic of high interest is connected to several problems and has a wide range of applications, see \cite{LoboVandenbergheBoyd98}. Since this latter
survey of applications many other important connections with physics, mechanics, economics, game theory, robotics, optimization and neural
networks have been found, such as the ones in \cite{ZhangZhangPan2013, YonekuraKanno2012, LuoAnXia2009, NishimuraHayashFukushima2009, AndreaniFriedlanderMelloSantos2008, KoChenYang2011, ChenTseng2005}.
If \( {\rm M}\)   is     symmetric, then the LSOCCP \eqref{eq:lclcp} is the optimality condition of the {\it  quadratic programming problem  under a second order cone 
constraint}, 
\begin{align}    \label{eq:napsc2}
 & \min   ~ \frac{1}{2}{\rm x}^\top {\rm M}{\rm x}+{\rm q}^\top {\rm x} + c \\  
  & \qquad  \qquad   \quad  {\rm x} \in {\cal{K}} \notag
\end{align}
where    \(c \) is a real number. Although not considered in this paper, it can be shown that any second order (in particular quadratic) conic optimization problem can 
be reformulated in terms of complementarity problems (in particular linear) related to the second order cone, see \cite{NemethZhang}.

We show that the semi-smooth Newton method  for solving problems \eqref{eq:pwls}, \eqref{eq:lclcp} and \eqref{eq:napsc2} has interesting
features,  for instance,  the global and  linear convergence of the generated sequence. Moreover, we  present  some computational experiments
designed to investigate its practical viability. For a given class of problem, our numerical results suggest that the number of required
iterations is almost unchanged.  The numerical results also indicate a remarkable  robustness   with respect  to the starting point.  

The organization of the paper is as follows.  In Section~\ref{sec:int.1}  some notations and auxiliary results  used in the paper, are presented. In particular,   important and useful properties of the projection mapping onto  the second order cone are studied.  In Section~\ref{sec: defscls}, we study the convergence properties of  the semi-smooth Newton method   for solving \eqref{eq:pwls}.  In Section~\ref{sec: defbp},  the results of Section~\ref{sec: defscls} are applied  to find a solution of \eqref{eq:lclcp}. In Section~\ref{sec:ctest}, we present  some computational experiments. Final remarks are considered in Section~\ref{sec:conclusions}.

%%%%%%%%%%%%%%%%%%%%%%%%%%%%%%
\subsection{Notations and preliminaries} \label{sec:int.1}
%%%%%%%%%%%%%%%%%%%%%%%%%%%%%%
In this section we present  the notations and  some auxiliary results used throughout the paper.
Let $\R^n$ be   the $n$-dimensional Euclidean space with  the canonical inner  product
$\lng\cdot,\cdot\rng$  and induced norm  $\|\cdot\|$.     If \(\alpha\in\R\), then denote \(\alpha^+:=\max\{\alpha,0\}\) and  \(\alpha^-:=\max\{-\alpha,0\}\). The set of all $m \times n$ matrices with real entries is denoted by $\R^{m \times n}$ and    $\R^n\equiv \R^{n \times 1}$.  The matrix ${\rm Id_n}$ denotes the $n\times n$ identity matrix.   Denote $\|{\rm E}\|:=\max \{\|{\rm E}{\rm x}\|~:~  {\rm x}\in \R^{n}, ~\|{\rm x}\|=1\}$ for any  \({\rm E} \in \R^{n\times n}\).  

\noindent The next useful result, known as Banach's Lemma, which was proved in  2.1.1,  page 32  of \cite{Ortega1990}.
\begin{lemma}[Banach's Lemma]\label{lem:ban}
Let $\emph{E} \in \R^{n\times n}$. If  $\|\emph{E}\|<1$,  then $\emph{E}-{\rm Id_n}$
is invertible and  $ \|(\emph{E}-{\rm Id_n})^{-1}\|\leq 1/\left(1-\|\emph{E}\|\right). $
\end{lemma}

We continue this section to present an interesting result on the eigenvalues of the sum of two symmetric matrices and its consequence. 

\begin{lemma} \label{l:wilkinson}
Let $A$ and $B$ be two $n \times n$ symmetric matrices. Denote the eigenvalues of $A$, $B$ and $A+B$ by $\lambda_i(A)$, $\lambda_i(B)$ and $\lambda_i( A+B)$  respectively,  where $i=1, \ldots, n$ and  all three sets are arranged in 
non-increasing order. Then $\lambda_i(A)+ \lambda_n(B) \leq \lambda_i(A+B)\leq \lambda_i (A)+ \lambda_1(B)$ for $i = 1, \ldots, n$.
\end{lemma} 
\begin{proof}
See page 101 of \cite{Wilkinson65}.
\end{proof}
\begin{corollary} \label{c:wilkinson}
Let $A_1, \ldots , A_p$  be  $n \times n$ symmetric matrices and $A = a _1A_1+ \cdots  + a_p A_p $ such that  $a _1+ \ldots  + a_p =1$ with $a_j \in [0,  1]$.  Denote the eigenvalues of $A_j$  and $A$ by $\lambda_{i}(A_j)$ and $\lambda_{i}(A)$  respectively,  where    $j=1, \ldots, p$, $i=1, \ldots, n$ and  all the sets are arranged in 
non-increasing order. Then $\emph{min} \{ \lambda_{n}(A_1),  \ldots,  \lambda_{n}(A_p) \} \leq \lambda_{i}(A) \leq   \max\{ \lambda_{1}(A_1),  \ldots, \lambda_{1}(A_p)\}$,  for $i = 1, \ldots, n$.
\end{corollary} 
\begin{proof}
Since  $a _1+ \ldots  + a_p =1$ with $a _j\in [0,  1]$ and the    eigenvalues  of     $a_j A_j$ are equals to   $a_j  \lambda_{i}(A_j)$ for   $j=1, \ldots, p$ and  $i=1, \ldots, n$. Thus,  the proof is an immediate consequence of  Lemma~\ref{l:wilkinson}.
\end{proof}
For a differentiable mapping $F: \R^{n} \to \R^{m}$ at a point ${\rm x} \in \R^{n}$, we denote by  $F'({\rm x}) \in \R^{m \times n}$ the Jacobian of $F$ at this point. If $F$ is a  locally Lipschitz continuous mapping, then the set 
$$
\partial_{B}F({\rm x}):=\left\{{\rm V}({\rm x})\in \R^{m \times n} ~:~  \exists ~ \{{\rm x}^k\}\subset D_F; ~ {\rm x}^k \to {\rm x}, ~F'({\rm x}^k) \to {\rm V({\rm x})}\right\},
$$
is nonempty and called the {\it B-subdifferential} of $F$ at ${\rm x}$, where $D_F\subseteq \R^{n} $ denotes the
set of points at which $F$ is differentiable. The convex hull of $
\partial_{B}F$, $\partial F({\rm x}):=\mathrm{conv}\, \partial_{B}F({\rm x})$, is known as
the {\it Clarke's generalized  Jacobian},  see   Definition~ 2.6.1 on page 70 of  \cite{Clarke1990}.   The next result is the generalization of the vector {\it Mean-Value Theorem}  for Lipschitz continuous mapping, see Proposition 2.6.5  on page 72  of    \cite{Clarke1990}. 
\begin{proposition} \label{pr:mvt}
Let $F: \R^{n} \to \R^{m}$ be a Lipschitz continuous mapping.  Then we have 
$$
F({\rm y})-F({\rm x}) \in \mathrm{conv}\, \partial F([{\rm x},  {\rm y}] )( {\rm y}- {\rm x}), \qquad  {\rm x},  {\rm y} \in   \R^{n},
$$
where the right hand side in the inclusion  denotes the convex hull  of all points of the form $U({\rm z}) ( {\rm y}- {\rm x})$  with $U({\rm z})\in \partial F({\rm z})$ and  ${\rm z}\in [ {\rm x}, {\rm y}]:=\{t {\rm x} + (1-t) {\rm y}~:~ t \in [0,  1]\}$. 
\end{proposition}
\noindent We end this section with the well-known contraction mapping principle, see 8.2.2,  page 153  of \cite{Ortega1990}.
\begin{theorem}  [Contraction mapping principle] \label{fixedpoint}
 Let  $\Phi : \R^{n} \to   \R^{n}$.  Suppose that there exists $\kappa \in  [0,1)$ such that  $\|\Phi({\rm y})-\Phi({\rm x})\| \leq \kappa \|{\rm y}-{\rm x}\|$, for all $ {\rm x}, {\rm y} \in  \R^{n}$. Then there exists a unique $\bar {\rm x}\in  \R^{n}$  such that $\Phi(\bar {\rm x}) = \bar {\rm x}$.
\end{theorem}
%%%%%%%%%%%%%%%%%%%%%%%%%%%%%%%%%%%%
\subsection{Properties of the projetion mapping} \label{sec:int.1.1}
%%%%%%%%%%%%%%%%%%%%%%%%%%%%%%%%%%%%
In this section we present some properties of the projection mapping onto a second order cone, which will play important roles in the study of equation in \eqref{eq:pwls}. We begin   with some notations. 

\noindent
 The {\it polar cone}  and the  {\it dual cone} of second order cone \({\cal{K}}\) in \eqref{eq:lc} are, respectively,  the sets
%\begin{equation} \label{eq:dpd}
 $${\cal{K}}^{\circ}:=\!\{ {\rm x}\in \R^n ~: ~\langle {\rm x}, {\rm y} \rangle\!\leq\! 0, ~  \forall \, {\rm y}\!\in\! {\cal{K}}\}, \qquad   {\cal{K}}^*\!\!:=\!\{ {\rm x}\in \R^n ~: ~\langle {\rm x}, {\rm y} \rangle\!\geq\! 0, ~ \forall \, {\rm y}\!\in\! {\cal{K}}\}.$$
%\end{equation}
Recall that  $ {\cal{K}}^{\circ}=-{\cal{K}}^*$ and  \({\cal{K}}\)  is self-dual, i.e., $ {\cal{K}}^*= {\cal{K}}$. Moreover,    it follows from 
 Moreau's  decomposition theorem, see \cite{Moreau1962} (see also  \cite[Theorem 3.2.5]{HiriartLemarecal1}),   that 
\begin{equation} \label{eq;cmdt}
{\rm x}={\rm P}_{\cal{K}}({\rm x})-{\rm P}_{\cal{K}}(-{\rm x}), \qquad \left\langle {\rm P}_{\cal{K}}({\rm x}), ~{\rm P}_{\cal{K}}(-{\rm x}) \right\rangle=0,   \qquad \qquad  ~{\rm x}\in\R^n.
\end{equation}
An explicit representation of the projection mapping    \({\rm P}_{\cal{K}}\)  onto \({\cal{K}}\) is given in the following result, see \cite[Proposition 3.3]{FukushimaTseng2002}.
\begin{lemma} \label{l:projude}
Let \({\rm x}:=(x_1, {{\rm x}_2})  \in \mathbb{R}\times {\mathbb{R}}^{n-1} \) and \({\cal{K}}\) be the second order cone. Then, 
 \begin{equation} \label{eq:projef}
{\rm P}_{\cal{K}}({\rm x})= \begin{cases}
\displaystyle  \frac{1}{2} \left( (x_1- \|{{\rm x}_2}\|)^+ + (x_1+ \|{{\rm x}_2}\|)^+ ,\, \left[ (x_1+ \|{{\rm x}_2}\|)^+ -(x_1- \|{{\rm x}_2}\|)^+\right] \displaystyle \frac{{\rm x}_2}{\|{\rm x}_2\|}\right), & {\rm x}_2\neq 0, \\
\\
\displaystyle \ \left(x_1^+,\,  0\right), & {\rm x}_2= 0.
\end{cases}
\end{equation}
\end{lemma}
\begin{remark} \label{re:ppo}
It is well-known that the orthogonal projection onto a closed convex set is continuous and firmly nonexpansive.  In particular,   $\|{\rm P}_{\cal{K}}({\rm x}) - {\rm P}_{\cal{K}}({\rm y})\|\leq \|{\rm x}-{\rm y}\|$ for all $ ~{\rm x}, {\rm y} \in\R^n,$  see \cite{HiriartLemarecal1}. 
\end{remark}

Since \({\rm P}_{\cal{K}}\)  is nonexpansive,  it is Lipschitz continuous. Thus,  by {\it Rademacher's Theorem}, we conclude that   \({\rm P}_{\cal{K}}\) is differentiable everywhere except at $x_1=\pm \|{\rm x_2}\|$. In the  next result we present the Jacobian  of  \({\rm P}_{\cal{K}}\) at the points where it is differentiable, and  as a consequence, we derive  some important properties of its  B-subdifferential.
\begin{lemma}  \label{l:dpj}
The projection mapping   \({\rm P}_{\cal{K}}\) onto the second order cone  \({\cal{K}}\)  is continuously differentiable at every ${\rm x}:=(x_1, {{\rm x}_2})  \in \mathbb{R}\times {\mathbb{R}}^{n-1} $  such that $x_1\neq \pm \|{{\rm x}_2}\|$ and   its  Jacobian    is given by
\begin{equation} \label{eq:dpjdp}
{\rm P}^{\prime}_{\cal{K}}({\rm x}):= \begin{cases}
{\rm Id_n},   & ~   x_1 > \|{\rm x}_2\|,\\
0,  &  ~ x_1< - \|{\rm x}_2\|, \\
\displaystyle \frac{1}{2}\begin{bmatrix}
1 & {\rm w}^\top \\
{\rm w} &  {\rm H}
\end{bmatrix},  & ~  - \|{\rm x}_2\| <x_1 <\|{\rm x}_2\|, 
\end{cases}
\end{equation}
 where  $ \displaystyle  {\rm w} ={\rm x}_2/\|{\rm x}_2\|$ and $ \displaystyle  H= \left[1+x_1/\|{\rm x}_2\|\right] {\rm Id_{n-1}} -  (x_1/\|{\rm x}_2\|){\rm w}{\rm w}^\top$. 
 As a consequence, at each  ${\rm x} \in  \mathbb{R}^n $, the matrix  $ {\rm V}({\rm x})\in \partial_{B} {\rm P}_{\cal{K}}({\rm x})$ has the following representation:
 
 \medskip
 
 \noindent {\bf (a)} If $x_1\neq \pm \|{{\rm x}_2}\|$, then  ${\rm V}({\rm x})= {\rm P}^{\prime}_{\cal{K}}({\rm x})$;
 
 \noindent {\bf (b)} If  ${{\rm x}_2}\neq 0$ and    $x_1= \|{{\rm x}_2}\|$,   then  ${\rm V}({\rm x})= {\rm Id_n}$  or 
 \begin{equation} \label{eq:x1enx2}
 {\rm V}({\rm x})= \frac{1}{2}\begin{bmatrix}
1 & {\rm w}^\top \\
{\rm w} & {\rm H}
\end{bmatrix}, 
\qquad  \displaystyle  {\rm w} =\frac{{\rm x}_2}{\|{\rm x}_2\|}, \qquad  \displaystyle  {\rm H}= 2{\rm Id_{n-1}} -  {\rm w}{\rm w}^\top;
 \end{equation}
 \noindent {\bf (c)} If  ${{\rm x}_2}\neq 0$ and    $x_1= -\|{{\rm x}_2}\|$,   then  ${\rm V}({\rm x})=0$ or 
 \begin{equation} \label{eq:x1emnx2}
 {\rm V}({\rm x})= \frac{1}{2}\begin{bmatrix}
1 & {\rm w}^\top \\
{\rm w} & {\rm H}
\end{bmatrix},
\qquad  \displaystyle  {\rm w} =\frac{{\rm x}_2}{\|{\rm x}_2\|}, \qquad  \displaystyle  {\rm H}=  {\rm w}{\rm w}^\top; 
\end{equation}
 \noindent {\bf (d)}  If  ${{\rm x}_2}= 0$ and    $x_1=0$,  then ${\rm V}({\rm x})=0$ or ${\rm V}({\rm x})={\rm Id_n}$ or ${\rm V}({\rm x})$ belongs to the set 
 $$
 \left\{   \frac{1}{2}\begin{bmatrix}
1 & {\rm w}^\top \\
{\rm w} & {\rm H}
\end{bmatrix} ~:~ {\rm H}=(1+\rho){\rm Id_{n-1}}-\rho{\rm w}{\rm w}^\top,  ~\mbox{for some} ~|\rho|<1 \; \mbox{and} ~\|{\rm w}\|=1\right\}.
 $$
Moreover, the eigenvalues of any matrix  ${\rm V}({\rm x})\in \partial_{B} {\rm P}_{\cal{K}}({\rm x})$ belong to the interval $[0 ,1]$, and consequently  $\|{\rm V}({\rm x})\|\leq 1$, for all ${\rm x} \in  \mathbb{R}^n $.
\end{lemma}
\begin{proof}
Combine Lemmas 2.5, 2.6 and 2.8 of \cite{KanzowFerencziFukushima2009}.
\end{proof}
In the next two lemmas, we obtain important properties of the orthogonal projection  ${\rm P}_{\cal{K}}$, which will be used in the definition and the convergence  analysis of the semi-smooth Newton method  for solving equation in \eqref{eq:pwls}. 
\begin{lemma}  \label{l:ppbx}
For every ${\rm x}:=(x_1, {{\rm x}_2})  \in \mathbb{R}\times {\mathbb{R}}^{n-1} $  and   $ {\rm V}({\rm x})\in \partial_{B} {\rm P}_{\cal{K}}({\rm x})$ there holds ${\rm V}({\rm x}){\rm x}= {\rm P}_{\cal{K}}({\rm x})$.
\end{lemma}
\begin{proof} Take any ${\rm x}=(x_1, {{\rm x}_2})  \in \mathbb{R}\times {\mathbb{R}}^{n-1} $. Then,
the analysis will be done following the four  possibilities from {\bf (a)} to {\bf (d)} as in the statement of Lemma~\ref{l:dpj}. First we assume {\bf (a)}, that is $x_1\neq \pm \|{{\rm x}_2}\|$. Then,    ${\rm V}({\rm x})= {\rm P}^{\prime}_{\cal{K}}({\rm x})$.   In this case, by using \eqref{eq:dpjdp},  we conclude that
$$
{\rm P}^{\prime}_{\cal{K}}({\rm x}){\rm x}= \begin{cases}
{\rm x},   & ~   x_1 > \|{\rm x}_2\|,\\
0,  &  ~ x_1< - \|{\rm x}_2\|, \\
\displaystyle \frac{1}{2}\left( x_1+ \|{{\rm x}_2}\|, ~ (x_1+ \|{{\rm x}_2}\|) \frac{{\rm x}_2}{\|{\rm x}_2\|} \right),  & ~  - \|{\rm x}_2\| <x_1 <\|{\rm x}_2\|.
\end{cases}
$$
On the other hand, by using \eqref{eq:projef}, we easily conclude that ${\rm P}^{\prime}_{\cal{K}}({\rm x}){\rm x}={\rm P}_{\cal{K}}({\rm x})$.

Next assume the conditions of {\bf (b)}, that is, ${{\rm x}_2}\neq 0$ and    $x_1= \|{{\rm x}_2}\|$. If ${\rm V}({\rm x})= {\rm Id_n}$, then
${\rm V}({\rm x}){\rm x}={\rm x}$. If ${\rm V}({\rm x})$ is given by \eqref{eq:x1enx2}, then 
$$
{\rm V}({\rm x}){\rm x}=  \frac{1}{2}\left(x_1+ \|{\rm x}_2\|, ~(x_1+ \|{\rm x}_2\|)\frac{{\rm x}_2}{\|x_2\|}   \right), 
$$
and using the assumption  $x_1= \|{{\rm x}_2}\|$  we conclude  ${\rm V}({\rm x}){\rm x}={\rm x}$.  Now, taking into account that  ${{\rm x}_2}\neq 0$,    $x_1- \|{{\rm x}_2}\|=0$ and   $x_1+ \|{{\rm x}_2}\|>0$   the equality in  \eqref{eq:projef}   becomes ${\rm P}_{\cal{K}}({\rm x})={\rm x}$. Therefore,  we also  have  ${\rm V}({\rm x})= {\rm P}_{\cal{K}}({\rm x})$.

Now we assume the conditions of {\bf (c)}, that is ${{\rm x}_2}\neq 0$ and    $x_1=- \|{{\rm x}_2}\|$.   If ${\rm V}({\rm x})= 0$ then  ${\rm V}({\rm x}){\rm x}=0$.    If ${\rm V}({\rm x})$ is given by  \eqref{eq:x1emnx2} then we have 
$$
{\rm V}({\rm x}){\rm x}=  \frac{1}{2}\left(x_1+ \|{\rm x}_2\|, ~(x_1+ \|{\rm x}_2\|)\frac{{\rm x}_2}{\|x_2\|}   \right), 
$$
and because  $x_1+ \|{{\rm x}_2}\|=0$,   we also conclude  that ${\rm V}({\rm x}){\rm x}=0$. Note that  the assumptions  ${{\rm x}_2}\neq 0$ and  $x_1=- \|{{\rm x}_2}\|$ imply $x_1+ \|{{\rm x}_2}\|=0$  and $x_1- \|{{\rm x}_2}\|<0$. Thus,   the equality in  \eqref{eq:projef}   implies  ${\rm P}_{\cal{K}}({\rm x})=0$.  Therefore,  we also  have  ${\rm V}({\rm x})= {\rm P}_{\cal{K}}({\rm x})$.

Finally,  we prove the statement of {\bf (d)}, assuming that ${{\rm x}_2}= 0$ and    $x_1=0$. It follows trivially that  ${\rm V}({\rm x}){\rm x}= 0$ and ${\rm P}_{\cal{K}}({\rm x})=0$.  Thus,  we also  have  ${\rm V}({\rm x}){\rm x}= {\rm P}_{\cal{K}}({\rm x})$. 

Therefore, we conclude that in all possible cases the lemma holds for every ${\rm x}=(x_1, {{\rm x}_2})  \in \mathbb{R}\times {\mathbb{R}}^{n-1} $, which  concludes the proof.
\end{proof}
\begin{lemma}  \label{l:error}
Let  \({\rm x}, {\rm y}\in \R^n\) and $ {\rm V}({\rm x})\in \partial_{B} {\rm P}_{\cal{K}}({\rm x})$. Then  \( \left\| {\rm P}_{\cal{K}}({\rm y}) - {\rm P}_{\cal{K}}({\rm x})-{\rm V}({\rm x})({\rm y}-{\rm x})\right\|\leq \|{\rm y}-{\rm x}\|\).
\end{lemma} 
\begin{proof}
After simple algebraic  manipulations, we conclude from Proposition~\ref{pr:mvt} that 
\begin{equation} \label{eq:avmt}
{\rm P}_{\cal{K}}({\rm y}) - {\rm P}_{\cal{K}}({\rm x}) - {\rm V}({\rm x})({\rm y}-{\rm x})  \in \mbox{conv}\, \partial {\rm P}_{\cal{K}}([{\rm x},  {\rm y}] )( {\rm y}- {\rm x})- {\rm V}({\rm x})({\rm y}-{\rm x}) , \qquad  {\rm x},  {\rm y} \in   \R^{n}.
\end{equation}
On the other hand,  Lemma~\ref{l:dpj} implies that  the eigenvalues of any matrix  ${\rm V}({\rm x})\in \partial_{B} {\rm P}_{\cal{K}}({\rm x})$
belong to the interval $[0,1]$, for all ${\rm x} \in  \mathbb{R}^n $.  Thus, by combining the definitions of $\partial {\rm P}_{\cal{K}} $  and  $ \mbox{conv}\, \partial {\rm P}_{\cal{K}}([{\rm x},  {\rm y}] )$ with Corollary~\ref{c:wilkinson},   we conclude that  the eigenvalues of any matrix  ${\rm U} \in  \mbox{conv}\, \partial {\rm P}_{\cal{K}}([{\rm x},  {\rm y}] ) $ also  belong to  $[0,1]$. Therefore, letting ${\rm U} \in  \mbox{conv}\, \partial {\rm P}_{\cal{K}}([{\rm x},  {\rm y}] ) $ and ${\rm V}({\rm x})\in \partial_{B} {\rm P}_{\cal{K}}({\rm x})$ and taking into account that  the eigenvalues of ${\rm U}$ and ${\rm V}({\rm x})$ belong to  $[0,1]$,  we conclude from    Lemma~\ref{l:wilkinson} that 
the eigenvalues of ${\rm U}-{\rm V}({\rm x})$ belong to  the interval $[-1,1]$. Hence, since ${\rm U}-{\rm V}({\rm x})$ is a symmetric matrix, we have 
$\|{\rm U}-{\rm V}({\rm x})\|\leq 1$.   On the other hand, \eqref{eq:avmt} implies that there exists ${\rm U}\in  \mbox{conv}\, \partial {\rm
P}_{\cal{K}}([{\rm x},  {\rm y}] )$ such that 
$$
{\rm P}_{\cal{K}}({\rm y}) - {\rm P}_{\cal{K}}({\rm x}) - {\rm V}({\rm x})({\rm y}-{\rm x})  =({\rm U}- {\rm V}({\rm x}))({\rm y}-{\rm x}).
$$
By taking the norm in the above equality and by using that $\|{\rm U}-{\rm V}({\rm x})\|\leq 1$,  we  get that the desired inequality holds.
\end{proof}

%%%%%%%%%%%%%%%%%%%%%%%%%%%%%%%%%%%
\section{A semi-smooth Newton method} \label{sec: defscls}
%%%%%%%%%%%%%%%%%%%%%%%%%%%%%%%%%%%

In this section, we  present  and  analyze  the semi-smooth Newton method for solving \eqref{eq:pwls}.  We begin by presenting an existence result of solutions for  equation \eqref{eq:pwls}.
\begin{proposition} \label{pr:uniqqpw} 
  If \(\left\|{\rm T}^{-1}\right\| <1 \)  then   \eqref{eq:pwls}   has a unique solution  for any $b\in \R^n$.
\end{proposition}
  \begin{proof}  The equation  \eqref{eq:pwls}   has a solution  if only if 
 \(
 \Phi ({\rm x})=-{\rm T}^{-1}{\rm P}_{\cal{K}}({\rm x}) + {\rm T}^{-1} b
 \)
 has a fixed point. It follows from definition of  \( \Phi\)  that
 \[
\Phi({\rm x})- \Phi({\rm y})= -{\rm T}^{-1} ({\rm P}_{\cal{K}}({\rm x}) - {\rm P}_{\cal{K}}({\rm y}) ), \qquad \qquad   ~{\rm x}, {\rm y} \in \mathbb{R}^n.
\]
Since   \(\left\|{\rm T}^{-1}\right\|<1 \) and  $\|{\rm P}_{\cal{K}}({\rm x}) - {\rm P}_{\cal{K}}({\rm y})\|\leq \|{\rm x}-{\rm y}\|$ for all $
~{\rm x}, {\rm y} \in\R^n$ and  after taking the norm in 
the last equality, we get \( \| \Phi({\rm x})- \Phi({\rm y}) \|\leq  \left\|{\rm T}^{-1}\right\| \|{\rm x}-{\rm y}\|,  \) for all   \( {\rm x}, {\rm y} \in \mathbb{R}^n\). Hence  \( \Phi\) is a contraction. 
Therefore,  by applying Theorem~\ref{fixedpoint}, we conclude that \(\Phi \) has  precisely an unique fixed point and consequently~\eqref{eq:pwls}     has  a unique solution.
\end{proof}
\noindent The next example shows that the bound  \(\left\|{\rm T}^{-1}\right\| <1 \) in Proposition~\ref{pr:uniqqpw} is strict.

\begin{example}
Consider equation~\eqref{eq:pwls} where
$$
\emph{T}=\begin{bmatrix}
1  &  0 \\
0 &   -1
\end{bmatrix} \qquad  \mbox{and}  \qquad 
\emph{b}=\begin{bmatrix}
 2\\
 0
\end{bmatrix}.
$$
Note that $\left\|{\rm T}^{-1}\right\|=1$. By using~\eqref{eq:projef}, direct calculations show us that $[1 \quad 1]^\top$ and $[1 \quad -1]^\top$ are solutions of~\eqref{eq:pwls}.
\end{example}

 The {\it semi-smooth Newton method}, introduced in   \cite{LiSun93},   for  finding the zero of the  semi-smooth function 
\begin{equation} \label{eq:fucpw}
F({\rm x}):={\rm P}_{\cal{K}}({ {\rm x}}) +{\rm T}{ {\rm x}}-b,   \qquad \qquad   ~{\rm x} \in \mathbb{R}^n, 
\end{equation}
with starting point   \({\rm x}^{0}\in \mathbb{R}^n\), it  is formally  defined by 
\begin{equation} \label{eq:nmqcpw}
F({\rm x}^{k})+ {\rm U}({\rm x}^k)\left({\rm x}^{k+1}-{\rm x}^{k}\right)=0,  \qquad   {\rm U}({\rm x}^k) \in \partial F({\rm x}^{k}), \qquad k=0,1,\ldots.
\end{equation}
Note that  \({\rm U}({\rm x}^k)\) is any subgradient in  \( \partial F({\rm x}^{k})\), the  Clarke generalized Jacobian of \(F\) at \({\rm
x}^{k}\). By  letting 
\begin{equation} \label{eq:mppw}
{\rm V}({\rm x})\in \partial_{B} {\rm P}_{\cal{K}}({\rm x}),  \qquad {\rm x}\in  \R^n, 
\end{equation} 
it is easy to see from \eqref{eq:fucpw} that  $ {\rm V}({\rm x}) +{\rm T}\in  \partial F({\rm x})$. Since Lemma~\ref{l:ppbx} implies that  \( {\rm
V}({\rm x}){\rm x}={\rm P}_{\cal{K}}({ {\rm x}}) \) for all \({\rm x}\in \R^n\) and $ {\rm V}({\rm x})\in \partial_{B} {\rm P}_{\cal{K}}({\rm
x})$, by    taking  ${\rm U}({\rm x}^k)={\rm V}({\rm x}^k) +{\rm T}$,    equation  \eqref{eq:nmqcpw}   becomes   
\begin{equation}\label{eq:newtonc2pw}
	\left[ {\rm V}({\rm x}^k) +{\rm T}\right]{\rm x}^{k+1}=b,  \qquad   {\rm V}({\rm x}^k)\in \partial_{B} {\rm P}_{\cal{K}}({\rm x}^k),  \qquad  \qquad k=0, 1,  \ldots , 
\end{equation}
which  formally defines the {\it semi-smooth Newton  sequence} \(\{{\rm x}^{k}\}\)    for solving  \eqref{eq:pwls}. It is worth  mentioning that a similar iteration was studied   in   \cite{BrugnanoCasulli2007}.   

\noindent The next proposition gives a stopping condition for the semi-smooth Newton iteration given in \eqref{eq:newtonc2pw}.

\begin{proposition} \label{pr:ft}
	If in \eqref{eq:newtonc2pw} ${\rm V}({\rm x}^{k+1})={\rm V}({\rm x}^{k})$,   then ${\rm x}^{k+1}$ is a solution of \eqref{eq:pwls}.
\end{proposition}
\begin{proof}
Since  ${\rm V}({\rm x}^{k+1})={\rm V}({\rm x}^{k})$, the  equation in  \eqref{eq:newtonc2pw}  gives us   
$
 \left[  {\rm V}({\rm x}^{k+1}) +{\rm T}\right]{\rm x}^{k+1}= b.
$
The equation  ${\rm V}({\rm x}^{k+1}){\rm x}^{k+1}={\rm P}_{\cal{K}}({\rm x}^{k+1})$, together with the previous one, yield
	\(
	{\rm P}_{\cal{K}}({\rm x}^{k+1})+{\rm T}{\rm x}^{k+1}= b, 
	\)
	which implies that  ${\rm x}^{k+1}$ is a  solution of   \eqref{eq:pwls}.
\end{proof}
The sufficient condition for the Q-linear convergence of the sequence generated by \eqref{eq:newtonc2pw} is presented in the next theorem. 
\begin{theorem}\label{th:mrqpw}
Let  $b \in \R^n$ and  ${\rm T} \in \R^{n\times n}$ be a nonsingular matrix.   Assume that \(\left\|{\rm T}^{-1}\right\|<1\). Then,
\eqref{eq:pwls}   has a unique solution \({\rm x}^*\in \mathbb{R}^n\) and,  for any starting point \({\rm x}^0 \in \R^{n}\),  the   semi-smooth Newton  sequence  \(\{{\rm x}^{k}\}\)  generated by \eqref{eq:newtonc2pw}   is well-defined. Additionally,  if
\begin{equation} \label{eq:CC2pw}
\left\|{\rm T}^{-1}\right\|<1/2, 
\end{equation}
then the sequence   \(\{{\rm x}^{k}\}\)    converges $Q$-linearly  to  \({\rm x}^*\) as follows:
  \begin{equation} \label{eq:lconv3pw}
\|{\rm x}^*-{\rm x}^{k+1}\|\leq \frac{\|{\rm T}^{-1}\|}{1- \|{\rm T}^{-1}\|} \|{\rm x}^*-{\rm x}^{k}\|,  \qquad k=0, 1, \ldots .
\end{equation}
\end{theorem}
\begin{proof} 
Let  \({\rm x} \in \mathbb{R}^n\).  It follows from Lemma~\ref{l:dpj} that $\|{\rm V}({\rm x})\|\leq 1$ for all ${\rm x} \in \R^n$. Hence, by using
the properties of the norm and by taking into account that  $\left\|{\rm T}^{-1}\right\|<1$,  we conclude that  \(\|{\rm T}^{-1}{\rm V}({\rm x})\|<1
\).  Thus,    Lemma~\ref{lem:ban}  implies that  \(-{\rm T}^{-1}{\rm V}({\rm x}) -{\rm Id_n}\) is nonsingular.  Since \({\rm T}\)  is
nonsingular and 
 \[
  {\rm V}({\rm x})+{\rm T}=-{\rm T}\left[ -{\rm T}^{-1}{\rm V}({\rm x})-{\rm Id_n}\right], \qquad  \quad ~x \in \mathbb{R}^n, 
 \]
 we obtain   that  ${\rm V}({\rm x})+{\rm T}$ is also nonsingular. Therefore, for any starting point \({\rm x}^0 \in \R^{n}\), \eqref{eq:newtonc2pw} implies that the sequence \(\{{\rm x}^{k}\}\) generated by \eqref{eq:newtonc2pw}   is well-defined.

 By using Proposition~\ref{pr:uniqqpw},  we conclude that \eqref{eq:pwls} has a unique solution \({\rm x}^*\in  \R^{n}\).  Since \({\rm x}^*\in  \R^{n}\) is the solution of \eqref{eq:pwls},   
 we have \( [{\rm V}({\rm x}^*)+{\rm T}]{\rm x}^*-b=0\), which together  with the definition of   $\{{\rm x}^{k}\}$ in \eqref{eq:newtonc2pw} and
 \eqref{eq:mppw}, implies 
\begin{align*}
{\rm x}^*-{\rm x}^{k+1}&=-[{\rm V}({\rm x}^{k})+{\rm T}]^{-1}\big[b- [{\rm V}({\rm x}^{k})+{\rm T}]{\rm x}^*\big]\\&= -[{\rm V}({\rm x}^{k})+{\rm T}]^{-1}\big[ [{\rm V}({\rm x}^*)+{\rm T}]{\rm x}^*-b- [{\rm V}({\rm x}^{k})+{\rm T}]{\rm x}^{k}+b- [{\rm V}({\rm x}^{k})+{\rm T}]({\rm x}^*-{\rm x}^{k})\big],  
\end{align*}
 for $k=0, 1, \ldots $. On the other hand, since \( {\rm V}({\rm x}){\rm x}={\rm P}_{\cal{K}}({\rm x}) \) for all \({\rm x}\in \R^n\), after  simple algebraic manipulations we obtain 
$$
[{\rm V}({\rm x}^*)+{\rm T}]{\rm x}^*-b - [{\rm V}({\rm x}^{k})+{\rm T}]{\rm x}^{k}+b- [{\rm V}({\rm x}^{k})+{\rm T}]({\rm x}^*-{\rm x}^{k})= {\rm P}_{\cal{K}}({\rm x}^*) - {\rm P}_{\cal{K}}({\rm x}^{k})-{\rm V}({\rm x}^{k})({\rm x}^*-{\rm x}^{k}),
$$
for $k=0, 1, \ldots $. By combining the above two equalities and by using the properties of the norm,   we obtain
\begin{equation}
\label{eq:new1}
\|{\rm x}^*-{\rm x}^{k+1}\|\leq \left \|[{\rm V}({\rm x}^{k})+{\rm T}]^{-1}\right\|  \left\| {\rm P}_{\cal{K}}({\rm x}^*) - {\rm P}_{\cal{K}}({\rm x}^{k})-{\rm V}({\rm x}^{k})({\rm x}^*-{\rm x}^{k})\right\|,  \qquad k=0, 1, \ldots .
\end{equation}
It follows from Lemma~\ref{l:error} that  $ \left\| {\rm P}_{\cal{K}}({\rm x}^*) - {\rm P}_{\cal{K}}({\rm x}^{k})-{\rm V}({\rm x}^{k})({\rm
x}^*-{\rm x}^{k})\right\| \leq \|{\rm x}^*-{\rm x}^{k}\|$, for \( k=0, 1, \ldots\), and, by combining the last inequality with \eqref{eq:new1}, we
get 
\begin{equation} \label{eq:conv}
\|{\rm x}^*-{\rm x}^{k+1}\|\leq  \left\|[{\rm V}({\rm x}^{k})+{\rm T}]^{-1}\right\|  \|{\rm x}^*-{\rm x}^{k}\|, \qquad k=0, 1, \ldots .
\end{equation}
On the other hand, by using the properties of the norm, after some simple algebraic manipulations,  we get
\[
\left\| [{\rm V}({\rm x}^{k})+{\rm T}]^{-1}\right\|=  \left\| [-{\rm T}^{-1}{\rm V}({\rm x}^{k})-{\rm Id_n}]^{-1} \left[-{\rm T}^{-1}\right]\right\|\leq  \left\| [{\rm T}^{-1}{\rm V}({\rm x}^{k})+{\rm Id_n}]^{-1}\right\| \|{\rm T}^{-1}\|,\quad k=0, 1, \ldots , 
\]
which combined with   Lemma~\ref{lem:ban} and \(\|{\rm T}^{-1}{\rm V}({\rm x}^{k})\|\leq \|{\rm T}^{-1}\|<1 \) implies  
\[
\left\| [{\rm V}({\rm x}^{k})+{\rm T}]^{-1}\right\| \leq \frac{ \|{\rm T}^{-1}\|}{1- \left\|{\rm T}^{-1}\right\|}, \qquad \quad k=0, 1, \ldots .
\]
Thus, the last inequality and \eqref{eq:conv} yield \eqref{eq:lconv3pw}. Note that \eqref{eq:CC2pw} implies  \(\|{\rm T}^{-1}\|/(1- \|{\rm
T}^{-1}\|)<1\). Therefore,   \eqref{eq:lconv3pw} implies that \(\{{\rm x}^{k}\}\) converges Q-linearly,  for any starting point \({\rm x}^0\),
to the 
solution \({\rm x}^*\) of  \eqref{eq:pwls}. Hence, the theorem is proven. 
\end{proof}
\noindent If ${\rm T} $ is a symmetric and positive definite matrix, then stronger results are obtained.
\begin{theorem} \label{th:defpos}
Let  $b \in \R^n$ and  ${\rm T} \in \R^{n\times n}$ be a symmetric and positive definite matrix.   
Then \eqref{eq:pwls} has a unique solution \({\rm x}^*\in \mathbb{R}^n\) and,  for any starting point \({\rm x}^0 \in \R^{n}\),  the   semi-smooth
Newton  sequence  \(\{{\rm x}^{k}\}\)  generated by \eqref{eq:newtonc2pw}   is well-defined. 
Moreover, if \(\left\|{\rm T}^{-1}\right\|<1\), then  \(\{{\rm x}^{k}\}\)    converges $Q$-linearly  to  \({\rm x}^*\) as follows:
$$\|{\rm x}^*-{\rm x}^{k+1}\|\leq \|{\rm T}^{-1}\| \|{\rm x}^*-{\rm x}^{k}\|,  \qquad k=0, 1, \ldots.$$
\end{theorem}
\begin{proof} 
Since ${\rm T}$ is  symmetric and positive definite, it follows from Lemma~\ref{l:wilkinson} that  ${\rm Id_n + T}$ is nonsingular. Thus,
taking into account that  ${\rm x} = {\rm P}_{\cal{K}}({\rm x})-{\rm P}_{\cal{K}}(-{\rm x})$, after some algebraic manipulations, we conclude that 
 \eqref{eq:pwls} is equivalent  to  ${\rm x} = \left[ {\rm Id_n + T}\right]^{-1}\left(b-{\rm P}_{\cal{K}}({-\rm x})\right)$. 
Therefore, equation~\eqref{eq:pwls} has a solution if only if 
$\Phi({\rm x})=\left[ {\rm Id_n + T}\right)]^{-1}\left(b-{\rm P}_{\cal{K}}({-\rm x})\right)$ has a fixed point. On the other hand, it follows from
the definition of $\Phi$ that
 \[
\Phi({\rm x})- \Phi({\rm y})= \left[ {\rm Id_n + T}\right]^{-1} (-{\rm P}_{\cal{K}}(-{\rm x}) + {\rm P}_{\cal{K}}(-{\rm y}) ), \qquad \qquad   ~{\rm x}, {\rm y} \in \mathbb{R}^n.
\]
Since ${\rm T}$ is symmetric and positive definite, it follows from Lemma~\ref{l:wilkinson} that 
$\left\|\left[ {\rm Id_n + T}\right]^{-1}\right\|=\kappa <1,$
where $\kappa = 1/(1+\lambda_{\rm min})$ and $\lambda_{\rm min}>0$ is the minimum eigenvalue of ${\rm T}$.
Now, proceeding as in Proposition~\ref{pr:uniqqpw}, it is possible to conclude that \( \Phi\) is a contraction and it has  
precisely a unique fixed point. Consequently   \eqref{eq:pwls} has  a unique solution.

Lemma~\ref{l:dpj} implies that the eigenvalues of ${\rm V(x)}$ belongs to the interval $[0,1]$, for all ${\rm x}\in\R^n$. 
Hence, the nonsingularity of ${\rm V(x)}+{\rm T}$ follows from Lemma~\ref{l:wilkinson}. As a consequence, the sequence \(\{{\rm x}^{k}\}\) generated
by \eqref{eq:newtonc2pw} is well-defined for any 
starting point.  In order to  prove the $Q$-linear convergence of  \(\{{\rm x}^{k}\}\) to  ${\rm x^*}\in\R^n$, the unique solution of
\eqref{eq:pwls},  we  proceed as in the proof of  Theorem~\ref{th:mrqpw} to obtain 
$$
\|{\rm x}^*-{\rm x}^{k+1}\|\leq  \left\|[{\rm V}({\rm x}^{k})+{\rm T}]^{-1}\right\|  \|{\rm x}^*-{\rm x}^{k}\|, \qquad k=0, 1, \ldots .
$$
Lemma~\ref{l:wilkinson} allows us to conclude  that $\left\|[{\rm V}({\rm x}^{k})+{\rm T}]^{-1}\right\|\leq \|{\rm T}^{-1}\|$. Thus, by combining the latter two 
inequalities we have
$$
\|{\rm x}^*-{\rm x}^{k+1}\|\leq  \|{\rm T}^{-1}\| \|{\rm x}^*-{\rm x}^{k}\|, \qquad k=0, 1, \ldots .
$$
Therefore, as we are under the assumption  \(\left\|{\rm T}^{-1}\right\|<1\), the last inequality implies that \(\{{\rm x}^{k}\}\) converges $Q$-linearly  to  \({\rm x}^*\in \mathbb{R}^n\),  for any starting point \({\rm x}^0\).
\end{proof}

The invertibility  of \({\rm V}({\rm x})+{\rm T}\),  for all \({\rm x}\in \mathbb{R}^n\),  is sufficient for the  well-definedness of the Newton method. However,   the next example shows that  
an additional condition on  ${\rm T}$ must be assumed for convergence, for instance,   \eqref{eq:CC2pw}.

\begin{example}
Consider the  function $F: \mathbb{R}^2 \to \mathbb{R}^2$ defined by  $F({\rm x})={\rm P}_ {\cal{K}}({\rm x})+{\rm T}{\rm x}-b$, where
$$
\emph{T}=\begin{bmatrix}
5  &  1 \\
1 &   0
\end{bmatrix} \qquad  \mbox{and}  \qquad 
\emph{b}=\begin{bmatrix}
 13\\
 3
\end{bmatrix}.
$$
Note that $\emph{T}$ is symmetric and $\|\emph{T}^{-1}\|=5.1926\ldots$. When the considered dimension is 2, ${\rm V}({\rm x})\in \partial_{B} {\rm P}_{\cal{K}}({\rm x})$ is equal to 
$$\begin{bmatrix} 0 & 0 \\0 & 0 \end{bmatrix},  \qquad \begin{bmatrix} 1 & 0 \\0 & 1 \end{bmatrix},  \qquad  \frac{1}{2} \begin{bmatrix} 1 & 1 \\1 & 1 \end{bmatrix}, \qquad \mbox{ or } \qquad \frac{1}{2} \begin{bmatrix} 1 & -1 \\-1 & 1 \end{bmatrix}.$$
Therefore, the matrices \( {\rm V}({\rm x})+\emph{T} \) are invertible,  for all \({\rm x}\in \mathbb{R}^2\). 
Moreover, ${\rm x}^*=[2 \quad 1]^\top$ is a zero of $F$.
By applying Newton method starting at ${\rm x}^0=[0 \quad 1]^\top$,  for finding the zeros of $F$, the generated sequence oscillates between the points
$$
{\rm x}^1=\displaystyle\begin{bmatrix}
4\\
-6
\end{bmatrix}, \qquad 
{\rm x}^2=\displaystyle\begin{bmatrix}
2\\
4 
\end{bmatrix}.
$$
It is useful to mention that, for all $k$, $x_1\neq \pm |{{\rm x}_2}|$. Therefore, ${\rm V}({\rm x}^k)=  {\rm P}^{\prime}_{\cal{K}}({\rm x}^k)$ and this example holds up for all different options in  items {\bf (b)}, {\bf (c)} and {\bf (d)} in {\rm Lemma \ref{l:dpj}}.
\end{example}
%%%%%%%%%%%%%%%%%%%%%%%%%%%%%%%%%%%
\section{Application to the  linear second order cone complementarity problem} \label{sec: defbp}
%%%%%%%%%%%%%%%%%%%%%%%%%%%%%%%%%%%
In this section,  we apply the results of Section~\ref{sec: defscls} to solve \eqref{eq:lclcp} and consequently  to find a solution of
\eqref{eq:napsc2}.  We begin by showing  that   from each solution of  the semi-smooth equation 
\begin{equation} \label{eq:lcalcp}
 [{\rm M}- {\rm Id_n}]{\rm P}_{\cal{K}}({\rm x}) +{\rm x}= -{\rm q}, 
\end{equation}
we obtain  a solution of  the  LSOCCP \eqref{eq:lclcp}:
\begin{proposition}  \label{pr:polscq}
If the vector \({\rm x}^*\)  is a solution of \eqref{eq:lcalcp}, then  \(  {\rm P}_{\cal{K}}({\rm x}^*) \)   is a solution of \eqref{eq:lclcp}.
\end{proposition}
\begin{proof}
It follows from \eqref{eq;cmdt}  that $ {\rm P}_ {\cal{K}}({\rm x}^*)-{\rm x}^*= {\rm P}_{\cal{K}}(-{\rm x}^*)$. Thus,  if \({\rm x}^*\in \R^n\) is a solution of \eqref{eq:lcalcp}, then
\begin{equation} \label{eq:epmq}
 {\rm M} {\rm P}_ {\cal{K}}({\rm x}^*) + {\rm q}= {\rm P}_{\cal{K}}(-{\rm x}^*).
\end{equation}
Since the second equality  in \eqref{eq;cmdt} implies that \(\left\langle   {\rm P}_ {\cal{K}}({\rm x}^*),  {\rm P}_ {\cal{K}}(-{\rm x}^*)
\right\rangle=0\) and since \({\rm P}_{\cal{K}}(-{\rm x}^*) \in {\cal{K}}\), the  equality \eqref{eq:epmq}   implies that 
\begin{equation}
\label{comp-cone}
 {\rm M} {\rm P}_ {\cal{K}}({\rm x}^*) + {\rm q} \in  {\cal{K}}, \qquad \left\langle   {\rm M} {\rm P}_ {\cal{K}}({\rm x}^*) + {\rm q}, ~  {\rm P}_ {\cal{K}}({\rm x}^*) \right\rangle =  0.
\end{equation}
Combining this with  \({\rm P}_ {\cal{K}}({\rm x}^*) \in{\cal{K}}$,   we conclude that  \({\rm P}_ {\cal{K}}({\rm x}^*)\)  is a solution of \eqref{eq:lclcp}  as claimed.  
\end{proof} 
\noindent The    {\it  semi-smooth Newton method  of starting point   \({\rm x}^{0}\in \mathbb{R}^n\) for solving \eqref{eq:lcalcp}},  is  given by 
 \begin{equation} \label{eq:newtonc2}
	 \left [ \left[ {\rm M} -{\rm Id_n}\right]{\rm V}({\rm x}^{k}) +{\rm Id_n}\right ]{\rm x}^{k+1}=-{\rm q},  \qquad   {\rm V}({\rm x}^k)\in \partial_{B} {\rm P}_{\cal{K}}({\rm x}^k),  \qquad k=0,1,\ldots . 
\end{equation}
\begin{remark} \label{eq:eqv}
 If    $\emph{M}-{\rm Id_n} $ is nonsingular, then letting  \({\rm T}=[\emph{M}-{\rm Id_n}]^{-1}\) and \(b=- {\rm T}{\rm q}\),  equation \eqref{eq:pwls}  becomes  \eqref{eq:lcalcp}.  As a consequence,   \eqref{eq:newtonc2pw}  turns into \eqref{eq:newtonc2}. Indeed, 
  \[
 {\rm x}^{k+1}= \left[ {\rm V}({\rm x}^{k}) +{\rm T}\right ]^{-1} b=\left[ \left[ {\rm M} -{\rm Id_n}\right]{\rm V}({\rm x}^{k}) +{\rm Id_n}\right ]^{-1}(-{\rm q}), \qquad \quad    {\rm V}({\rm x}^k)\in \partial_{B} {\rm P}_{\cal{K}}({\rm x}^k), 
\]
for $k=0,1,\ldots$, which is equivalent to  the  semi-smooth Newton method defined in   \eqref{eq:newtonc2}.
\end{remark}
\noindent In the following two results we present   sufficient existence and uniqueness conditions for \eqref{eq:lcalcp}.
\begin{proposition}\label{pr:uniqqM}  
If \(\left\| \emph{M} -{\rm Id_n}\right\|<1 \)  then  \eqref{eq:lcalcp}  has a unique solution,   for any ${\rm q}\in \R^n$.
\end{proposition}
\begin{proof} 
First note that,   \eqref{eq:lcalcp} has a unique  solution if, and only if, 
$\Phi({\rm x})= -\left[ {\rm M} - {\rm Id_n} \right]{\rm P}_{\cal{K}}({\rm x}) - {\rm q}$ has a unique fixed point. The definition of $\Phi$  implies  that
 \[
\Phi({\rm x})- \Phi({\rm y})=   \left[ {\rm M} - {\rm Id_n} \right] \left({\rm P}_{\cal{K}}({\rm y}) - {\rm P}_{\cal{K}}({\rm x}) \right), \qquad \qquad   ~{\rm x}, {\rm y} \in \mathbb{R}^n.
\]
Hence, from  Remark~\ref{re:ppo} we obtain  that $ \| \Phi({\rm x})- \Phi({\rm y})\| \leq \left\|{\rm M} - {\rm Id_n}\right\| \|{\rm x}-{\rm
y}\|$.  Since   $\left\|{\rm M} - {\rm Id_n}\right\| <1$,   $\Phi$   is a contraction.   Therefore,   from Theorem~\ref{fixedpoint} we
conclude that  $\Phi$   has a unique fixed point,   for any ${\rm q}\in \R^n$, which implies that \eqref{eq:lcalcp} has a unique  solution. 
\end{proof}
\begin{proposition}\label{pr:uniqqMI}  
If   $\emph{M}$ is nonsingular and  \(\left\| \emph{M}^{-1} -{\rm Id_n}\right\|<1 \),   then  \eqref{eq:lcalcp}  has a unique solution,   for any ${\rm q}\in \R^n$.
\end{proposition}
\begin{proof} 
Since ${\rm M}$ is nonsingular, by taking into account the first equality in \eqref{eq;cmdt}, after some algebraic manipulations we can conclude   that  \eqref{eq:lcalcp} is equivalent to 
\begin{equation} \label{eq-m-1M}
\left[{\rm M}^{-1} -{\rm Id_n}\right]{\rm P}_{\cal{K}}({\rm -x})-{\rm x}={\rm M}^{-1}{\rm q}.
\end{equation}
Define the auxiliary function $\Theta({\rm x})= \left[{\rm M}^{-1} -{\rm Id_n}\right]{\rm P}_{\cal{K}}(-{\rm x}) - {\rm M}^{-1}{\rm q}$.  Note
that  \eqref{eq-m-1M} has a unique solution if, and only if,  $\Theta$ has a unique fixed point. On the other hand, the definition of $\Theta$  implies  
\[
\Theta({\rm x})- \Theta({\rm y})=   \left[{\rm M}^{-1} -{\rm Id_n}\right]\left({\rm P}_{\cal{K}}(-{\rm x}) - {\rm P}_{\cal{K}}(-{\rm y}) \right), \qquad \qquad   ~{\rm x}, {\rm y} \in \mathbb{R}^n.
\]
It follows from Remark~\ref{re:ppo} that  $ \| \Theta({\rm x})- \Theta({\rm y}) \| \leq \left\|{\rm M}^{-1} -{\rm Id_n}\right\| \|{\rm x}-{\rm y}\|$ and,  due to   $\left\|{\rm M}^{-1} -{\rm Id_n}\right\| <1$,  we conclude that  $\Theta$   is a contraction. Therefore,   Theorem~\ref{fixedpoint}   implies that $\Theta$   has a unique fixed point,   for any ${\rm q}\in \R^n$. Consequently,  \eqref{eq:lcalcp}  has a unique solution, for any ${\rm q}\in \R^n$.
\end{proof}
\begin{remark} \label{Rpr:uniqqM}  
Note that,  there exist symmetric matrices for which neither $\left\|\emph{\rm M} -{\rm Id_n}\right\|<1$, nor $\left\|\emph{\rm M}^{-1} -{\rm
Id_n}\right\|<1$ are satisfied. For example, such a matrix is 
\begin{equation}\label{exemplo-2}
{\rm M}=\left[\begin{matrix}
1/3 &0\\0&3
\end{matrix}\right].
\end{equation}	
\end{remark}	
In the next theorem a convergence result  for the semi-smooth Newton  sequence  \(\{{\rm x}^{k}\}\),  generated by \eqref{eq:newtonc2},  is presented.
\begin{theorem}\label{theorem-4}
Let  $q \in \R^n$ and  ${\rm M} \in \R^{n\times n}$ be a symmetric matrix. Assume that $\emph{M}-{\rm Id_n} $ is   nonsingular and  \(\left\| \emph{M} -{\rm Id_n}\right\|<1\). Then,   \eqref{eq:lcalcp}   has a unique solution \({\rm x}^*\in \mathbb{R}^n\) and,  for any starting point \({\rm x}^0 \in \R^{n}\),  the     sequence  \(\{{\rm x}^{k}\}\)  generated by \eqref{eq:newtonc2}   is well-defined. Additionally,  if
%\begin{equation} \label{eq:CC2pwM}
$\left\| \emph{M} -{\rm Id_n}\right\|<1/2$, 
%\end{equation}
then  \(\{{\rm x}^{k}\}\)    converges $Q$-linearly  to  \({\rm x}^*\in \mathbb{R}^n\),  the unique solution  of  \eqref{eq:pwls}, as follows:
%  \begin{equation} \label{eq:lconv3pw}
$$\|{\rm x}^*-{\rm x}^{k+1}\|\leq \frac{\left\| \emph{M} -{\rm Id_n}\right\|}{1- \left\| \emph{M} -{\rm Id_n}\right\|} \|{\rm x}^*-{\rm x}^{k}\|,  \qquad k=0, 1, \ldots .$$
%\end{equation}
Moreover,  \( {\rm P}_ {\cal{K}}({\rm x}^*)\)   is a solution of \eqref{th:defpos}.
\end{theorem}
\begin{proof} 
The proof follows by combining    Proposition~\ref{pr:uniqqM}, Remark~\ref{eq:eqv}, Theorem~\ref{th:mrqpw}, and Proposition~\ref{pr:polscq}.
\end{proof}
If we assume \({\rm M}\) is a symmetric and positive definite in Theorem  \ref{theorem-4} then stronger results are obtained.    We begin showing that  the semi-smooth Newton method  in \eqref{eq:newtonc2} is always well-defined. 
\begin{lemma}\label{nonsingGC}
If  \({\rm M}\) is symmetric and positive definite, then the following    matrix is nonsingular 
\begin{equation} \label{eq:mnm}
\left[ \emph{M} -{\rm Id_n}\right]{\rm V}({\rm x}) +{\rm Id_n}, \qquad  \quad ~{\rm x} \in \mathbb{R}^n. 
\end{equation}
 As a consequence,  the  semi-smooth Newton sequence \(\{{\rm x}^{k}\}\) generated by \eqref{eq:newtonc2}    is well-defined,  for any starting point \({\rm x}^{0} \in \R^{n}\).
\end{lemma}
\begin{proof}
To simplify the notations let \( {\rm V}={\rm V}({\rm x})\).  Let us suppose, by contradiction, that the matrix in \eqref{eq:mnm} is singular. Thus  there exists \({\rm u}\in \mathbb{R}^n\) such that 
\[
\left(\left[ {\rm M}-{\rm Id_n}\right]{\rm V}+{\rm Id_n}\right){\rm u}= 0,  \qquad  {\rm u} \neq 0.
\]
It is straightforward to see that the last equality is equivalent to the following one
\begin{equation} \label{eq:nsq}
{\rm M}{\rm V}{\rm u} = \left[{\rm V}- {\rm Id_n}\right]{\rm u},  \qquad  {\rm u} \neq 0.
\end{equation}
Since \({\rm M}\) is symmetric and positive definite,  there exists  a  nonsingular matrix \({\rm L }\in \R^{n\times n}\) such that \({\rm M}={\rm L }{\rm L }^\top\). Taking into account  that    \({\rm M}={\rm L }{\rm L }^\top\) and that Lemma~\ref{l:dpj} implies ${\rm V}={\rm V}^\top $ and  $ \|{\rm V}\|\leq 1$, the  equality in  \eqref{eq:nsq}  easily implies that 
\[
\left\| {\rm L }^\top {\rm V}{\rm u}\right\|^2 = \left\langle {\rm V} {\rm M} {\rm V}{\rm u},  {\rm u} \right\rangle=   \left\langle   ({\rm V}^2-{\rm V}){\rm u}, {\rm u}\right\rangle  \leq 0. 
\]
Thus,   \( {\rm L }^\top {\rm V}{\rm u} = 0\).  Because \({\rm M}={\rm L }{\rm L }^\top\) and  \( {\rm L }^\top {\rm V}{\rm u} = 0\),  equality  \eqref{eq:nsq} implies  that \(({\rm V} -{\rm Id_n}){\rm u} = 0\), or equivalently, \({\rm V}{\rm u}={\rm u}\). Hence, 
\[
 {\rm L }^\top {\rm u} =  {\rm L }^\top {\rm V}{\rm u} = 0,   \qquad  {\rm u} \neq 0,
\]
which contradicts the nonsingularity of ${\rm L }$. Therefore, the matrix in \eqref{eq:mnm} is nonsingular for all \({\rm x}\in \mathbb{R}^n\)  and the first part of the lemma is proven.
	
To prove 	the second part of the lemma, combine  the formal definition of  \(\{{\rm x}^k\}\) in \eqref{eq:newtonc2} and the first part of this lemma.
\end{proof}
\begin{theorem} \label{th:defposM}
Let  ${\rm q} \in \R^n$ and  ${\rm M} \in \R^{n\times n}$ be a symmetric and positive definite matrix.  Then,  for any starting point \({\rm x}^0 \in \R^{n}\),  the   semi-smooth Newton  sequence  
\(\{{\rm x}^{k}\}\)  generated by \eqref{eq:newtonc2}   is well-defined.  Additionally,  if  $\emph{M}-{\rm Id_n} $ is  positive definite and \(\left\| \emph{M} -{\rm Id_n}\right\|<1\),
then \eqref{eq:lcalcp} has a unique solution  \({\rm x}^*\in \mathbb{R}^n\) and  \(\{{\rm x}^{k}\}\)    converges $Q$-linearly  to  \({\rm x}^*\) as follows 
$$
\|{\rm x}^*-{\rm x}^{k+1}\|\leq \left\| \emph{M} -{\rm Id_n}\right\| \|{\rm x}^*-{\rm x}^{k}\|,  \qquad k=0, 1, \ldots.
$$
Moreover,  \( {\rm P}_ {\cal{K}}({\rm x}^*)\)   is a solution of \eqref{th:defpos}.
\end{theorem}
\begin{proof} 
The proof follows by combining   Proposition~\ref{pr:uniqqM},  Lemma~\ref{nonsingGC},     Remark~\ref{eq:eqv},  Theorem~\ref{th:defpos} and Proposition~\ref{pr:polscq}.
\end{proof}
%%%%%%%%%%%%%%%%%%%%%%%%%%%%%%%%%%%%%%%%%%%%%%%%%%
From now on,  we will consider a parametric version of equation \eqref{eq:lcalcp}, which will be specially useful to study the  second order cone linear  complementarity problem, whenever  ${\rm M}$ is positive definite. Let  $q \in \R^n$ and  ${\rm M} \in \R^{n\times n}$ be a symmetric matrix defining  equation in  \eqref{eq:lcalcp}. Let  $\beta >0$ and define ${\rm M_\beta}:=\beta {\rm M}$ ,  $ {\rm  q_\beta }:= {\rm \beta q }$ and  consider the   {\it parametric auxiliary equation} 
\begin{equation} \label{eq:lcalcp*}
[{\rm M_\beta}- {\rm Id_n}]{\rm P}_{\cal{K}}({\rm y}) +{\rm y}= -{\rm  q_\beta }.
\end{equation} 
Note that  the last equation has the same algebraic  structure  of equation in  \eqref{eq:lcalcp}.  In the next  remark we point out some interesting properties of  \eqref{eq:lcalcp*}, which are analogous properties of the equation in \eqref{eq:lcalcp}.
\begin{remark} \label{r:eqeq}
It is worth mentioning that the result of {\rm Proposition \ref{pr:polscq}} remains true if equation \eqref{eq:lcalcp} is replaced by
\eqref{eq:lcalcp*}.  In other words, if ${\rm y}^*$ is solution of \eqref{eq:lcalcp*}, then ${\rm P}_{\cal{K}}({\rm y}^*)$ is a solution of
\eqref{eq:lclcp}. The proof of this statement follows from the same idea as in the proof of {\rm Proposition \ref{pr:polscq}}, by noting that,   due to
$\cal{K}$  being a cone,  the equation  \eqref{comp-cone} still holds for  ${\rm M}={\rm M_\beta}$ and $ {\rm q}={\rm  q_\beta }$.    Moreover,  if
\(\left\| {\rm M_\beta} -{\rm Id_n}\right\|<1 \) or \(\left\| {\rm M_\beta}^{-1} -{\rm Id_n}\right\|<1 \), then equation in  \eqref{eq:lcalcp*}
has also a unique solution.  Indeed, the result follows by applying  {\rm Propositions   \ref{pr:uniqqM}} and {\rm \ref{pr:uniqqMI}} with    ${\rm
M}={\rm M_\beta}$ and $ {\rm q}={\rm  q_\beta }$.
\end{remark}
Now   we are going to show the advantage to choose an appropriate parameter $\beta>0$ in \eqref{eq:lcalcp*} instead of taking $\beta=1$ as in equation \eqref{eq:lcalcp}.
We will begin with the following remark.
\begin{remark} \label{r:wdef}
Additionally,  if  ${\rm M}$  is positive definite and  $0\leq \beta<2/\|{\rm M}\|$, which includes the simple example of {\rm Remark
\ref{Rpr:uniqqM}}, then equation \eqref{eq:lcalcp*} always has a unique solution. Actually,  if  ${\rm M}$  is positive definite and $0<
\beta<2/\|{\rm M}\|$, then we have
$$
\left\|{\rm M_\beta} -{\rm Id_n}\right\|=\left\|{\rm \beta M} -{\rm Id_n}\right\|<1,
$$
and, by applying  {\rm Proposition \ref{pr:uniqqM}}  with    $M={\rm M_\beta}$ and $ {\rm q}={\rm  q_\beta }$,  we conclude that
\eqref{eq:lcalcp*} has a unique solution.  In particular, note that, by  replacing the matrix in \eqref{exemplo-2}, which has the norm equal to
$3$, with the matrix 
\begin{equation*}\label{exemplo-2.beta}
{\rm M_\beta}:=\left[\begin{matrix}
\beta/3 &0\\0&3\beta
\end{matrix}\right], \quad \qquad 0< \beta<1/3,
\end{equation*}
we have $\left\|{\rm M_\beta} -{\rm Id_n}\right\|<1$ and in this case  we conclude that \eqref{eq:lcalcp*} has a unique solution.
\end{remark}
\noindent The   semi-smooth Newton method  for solving \eqref{eq:lcalcp*}, with starting point   \({\rm y}^{0}\in \mathbb{R}^n\),  is  given by 
\begin{equation} \label{eq:newtonc2*}
\left [ \left[ {\rm M_\beta} -{\rm Id_n}\right]{\rm V}({\rm y}^{k}) +{\rm Id_n}\right ]{\rm y}^{k+1}=-\beta{\rm q},  \qquad   {\rm V}({\rm y}^k)\in \partial_{B} {\rm P}_{\cal{K}}({\rm y}^k),  \qquad k=0,1,\ldots. 
\end{equation}
The next result shows how to take advantage of choosing an appropriate parameter  $\beta>0$,  in order to apply the semi-smooth Newton method for obtaining  a solution of 
\eqref{eq:lclcp}.
\begin{theorem}\label{theorem-6}
Let  $q \in \R^n$ and  ${\rm M} \in \R^{n\times n}$ be a symmetric positive definite matrix. Then,  for any starting point \({\rm y}^0 \in \R^{n}\),  the   semi-smooth Newton  sequence  
\(\{{\rm y}^{k}\}\)  generated by \eqref{eq:newtonc2*}   is well-defined.  Moreover,  if  $0< \beta<2/\|{\rm M}\|$  then   \eqref{eq:lcalcp*}  has a unique solution \({\rm y}^*\in \mathbb{R}^n\).   Additionally, if $ \|{\rm M}\|\|{\rm M}^{-1}\|< 3$ and 
\begin{equation}\label{AssumRate}
\frac{1}{2}\|{\rm M}^{-1}\|<\beta< \frac{3}{2}  \frac{1}{\|{\rm M}\|}, 
\end{equation}
then  the sequence \(\{{\rm y}^{k}\}\)    converges $Q$-linearly  to  \({\rm y}^*\in \mathbb{R}^n\),  the unique solution  of  \eqref{eq:lcalcp*}, as follows:
\begin{equation}\label{paraRate}
\|{\rm y}^*-{\rm y}^{k+1}\|\leq \frac{\left\| \emph{M}_\beta -{\rm Id_n}\right\|}{1- \left\| \emph{M}_\beta -{\rm Id_n}\right\|} \|{\rm y}^*-{\rm y}^{k}\|,  \qquad k=0, 1, \ldots .
\end{equation}
Furthermore,  \( {\rm P}_ {\cal{K}}({\rm y}^*)\)   is a solution of \eqref{th:defpos}.
\end{theorem}
\begin{proof} 
By using the same idea as in  Lemma~\ref{nonsingGC},  we  can prove that   $ \left[ {\rm M_\beta} -{\rm Id_n}\right]{\rm V}({\rm y}^{k}) +{\rm Id_n}$  is a  nonsingular  matrix, for $k=0,1, \ldots$. Consequently, for any starting point \({\rm y}^0 \in \R^{n}\),  the   semi-smooth Newton  sequence   \(\{{\rm y}^{k}\}\)  generated by \eqref{eq:newtonc2*}   is well-defined. Now, assuming that  $0< \beta<2/\|{\rm M}\|$, we conclude from  Remark~\ref{r:wdef} that the  equation in  \eqref{eq:lcalcp*}  has a unique solution \({\rm y}^*\in \mathbb{R}^n\).  Then,  
 $$\left [ \left[ {\rm M_\beta} -{\rm Id_n}\right]{\rm V}({\rm y}^{*}) +{\rm Id_n}\right ]{\rm y}^*=-\beta {\rm q},$$ which, together  with
 definition of   $\{{\rm y}^{k}\}$ in \eqref{eq:newtonc2*} and  \( {\rm V}({\rm x}){\rm x}={\rm P}_{\cal{K}}({\rm x}) \) for all \({\rm x}\in
 \R^n\),  yield
 \begin{align*}
 {\rm y}^{k+1}- {\rm y}^*&= \left [ \left[ {\rm M_\beta} -{\rm Id_n}\right]{\rm V}({\rm y}^{k}) +{\rm Id_n}\right ]^{-1} \left[ \left [ \left[ {\rm M_\beta} -{\rm Id_n}\right]{\rm V}({\rm y}^{*}) +{\rm Id_n}\right ]{\rm y}^*-\left [ \left[ {\rm M_\beta} -{\rm Id_n}\right]{\rm V}({\rm y}^{k}) +{\rm Id_n}\right ]{\rm y}^*\right]\\&=\left [ \left[ {\rm M_\beta} -{\rm Id_n}\right]{\rm V}({\rm y}^{k}) +{\rm Id_n}\right ]^{-1} \left[ {\rm M_\beta} -{\rm Id_n}\right]\left[
 {\rm P}_{\cal{K}}\left( {\rm y}^*\right) -  {\rm P}_{\cal{K}}( {\rm y}^k)-{\rm V}( {\rm y}^k) \left( {\rm y}^*-  {\rm y}^{k}\right)\right], 
 \end{align*}
 for $k=0, 1, \ldots $. By combining   Lemma~\ref{l:error}  with this last equality  and by using the properties of the norm, we have
 \begin{equation}\label{dos-term}
 \|{\rm y}^*-{\rm y}^{k+1}\|\leq \left \|\left [ \left[ {\rm M_\beta} -{\rm Id_n}\right]{\rm V}({\rm y}^{k}) +{\rm Id_n}\right ]^{-1}\right\|  \left\| {\rm M_\beta} -{\rm Id_n}\right\| \| {\rm y}^*-  {\rm y}^{k}\|.
 \end{equation} 
 On the other hand,  for  $0< \beta<2/\|{\rm M}\|$, we have
 $
 \left\|\left[ {\rm M_\beta} -{\rm Id_n}\right]{\rm V}({\rm y}^{k})\right\|\le \left\| {\rm M_\beta} -{\rm Id_n}\right\|<1.
 $
 Thus,  by using Lemma \ref{lem:ban}, we have
 $$
 \left \|\left [ \left[ {\rm M_\beta} -{\rm Id_n}\right]{\rm V}({\rm y}^{k}) +{\rm Id_n}\right ]^{-1}\right\|\le \frac{1}{1-\left\| {\rm M_\beta} -{\rm Id_n}\right\|}, 
 $$
which combined with  \eqref{dos-term}  gives us \eqref{paraRate}. Moreover, by using assumption \eqref{AssumRate}, we conclude 
 $$
 \frac{\left\| {\rm M_\beta} -{\rm Id_n}\right\|}{1-\left\| {\rm M_\beta} -{\rm Id_n}\right\|}<1.
 $$
 Then, the last inequality, together with \eqref{paraRate}, imply that   $\{{\rm y}^k\}$ converges Q-linearly to ${\rm y}^*$  and, by using Remark~\ref{r:eqeq}, we obtain  that  \( {\rm P}_ {\cal{K}}({\rm y}^*)\)  is a solution of  \eqref{th:defpos}.
\end{proof}
We end this section by presenting an upper bound for the rate of convergence of the semi-smooth Newton method in \eqref{eq:newtonc2*}, which
depends only of the minimum and maximum eigenvalues of ${\rm M}$.
\begin{remark} Let us focus our attention on the convergence rate of $\{{\rm y}^k\}$,  the sequence generated by the semi-smooth Newton method in
	\eqref{eq:newtonc2*},  when ${\rm M}$ is symmetric and positive definite. The inequality in  \eqref{paraRate}  shows  that $\|{\rm
	M_\beta-Id_n}\|$ determines the rate, which depends on $\beta$. Indeed, the upper bound for the rate of convergence is
$$
r(\beta):=\frac{\left\| { \beta \rm M} -{\rm Id_n}\right\|}{1-\left\| {\beta \rm M} -{\rm Id_n}\right\|}, \qquad  \frac{1}{2}\|{\rm M}^{-1}\|<\beta< \frac{3}{2}  \frac{1}{\|{\rm M}\|}.
$$ 
Now, we are going to compute the  minimum  value of the function $r$ in the range of $\beta$ given above. Since the function  $t \mapsto t/(1-t)$
is increasing, the  minimum  value of  $r$ in this range  is reached when
\begin{equation}\label{normB*}
\beta_*=\ds{\rm argmin}\left\{ \left\| { \beta \rm M} -{\rm Id_n}\right\| ~: ~\frac{1}{2}\|{\rm M}^{-1}\|<\beta< \frac{3}{2}  \frac{1}{\|{\rm M}\|} \right\}.
\end{equation}
Let $\lambda_{\rm min}$ and $\lambda_{\rm max}$ be the minimum and the maximum eigenvalues of ${\rm M}$, respectively. Thus,  since ${\rm M}$ is
symmetric and positive definite $\|{\rm M}^{-1}\|=1/\lambda_{\rm min}$ and $\lambda_{\rm max}=\|{\rm M}\|$, and moreover, by using \eqref{normB*}, we
obtain
$$
 \beta_*=\ds{\rm argmin} \left\{  \max\left\{|\beta \lambda_{\rm min}-1|, |\beta \lambda_{\rm max}-1|\right\}~: ~\frac{1}{2}\frac{1}{ \lambda_{\rm min}}<\beta< \frac{3}{2}  \frac{1}{\lambda_{\rm max}} \right\}.
 $$
 Some calculations show that 
$$
\beta_*=\frac{2}{\lambda_{\rm max}+\lambda_{\rm min}}, \qquad  \quad 
r(\beta_*)=\frac{\lambda_{\rm max}-\lambda_{\rm min}}{2\lambda_{\rm min}}.
$$
Additionally, if  $\lambda=\lambda_{\rm max}=\lambda_{\rm min}$, then $\beta_*=1/\lambda$ and  $r(\beta_*)=0$. Thus, the inequality in
\eqref{paraRate} implies that ${\rm y}^1={\rm y}^*$. Hence the sequence $\{{\rm y}^k\}$ generated by \eqref{eq:newtonc2*} with $\beta=\beta_*$
converges to ${\rm y}^*$ in just one iteration.
\end{remark}
%%%%%%%%%%%%%%%%%%%%%%%%%%
\section{Computational Results} \label{sec:ctest} 
We implemented the semi-smooth Newton method \eqref{eq:newtonc2pw} for solving equation \eqref{eq:pwls} in Matlab 7.11.0.584 (R2010b). 
When the projection mapping   \({\rm P}_{\cal{K}}\) onto the second order cone  \({\cal{K}}\)  is not continuously differentiable at ${\rm x}\in \mathbb{R}^n$, we define  
${\rm V}({\rm x})\in \partial_{B} {\rm P}_{\cal{K}}({\rm x})$ in the simplest way. This means that ${\rm V}({\rm x})$ is equal to $\rm Id_n$ in case of item (a) and the null matrix in cases of items (b) and 
(c) of Lemma~\ref{l:dpj}.
The method stops at the iterate ${\rm x}^k\in \R^n$ reporting ``Solution found'' if $\left\|{\rm P}_{\cal{K}}({ {\rm x}^k}) +{\T} { {\rm x}^k}-b\right\| \leq 10^{-6}$.
Failure is considered when the number of iterations exceeds $20$.
All codes are freely available at \url{https://orizon.mat.ufg.br/p/3374-links}.
The experiments were run on a 3.4 GHz Intel(R) i7 with 4 processors, 8Gb of RAM, and Linux operating system.

In order to verify the applicability of our approach, we tested the semi-smooth Newton method \eqref{eq:newtonc2pw} in several random problems \eqref{eq:pwls}. 
A linear system must be solved in each iteration of the method.
For this purpose, we used the {\it mldivide} (same as {\it backslash}) command of Matlab. 
Following, we enlighten how the problems data were generated.

\noindent {\bf (i)} {\it Matrix ${\rm T} $}: We consider cases where the matrix ${\rm T}$ is dense and cases where ${\rm T}$ is sparse for different dimension values $n$.
In the first case, we randomly generated the fully dense matrix ${\rm T}$ from a uniform distribution on $(-10,10)$. 
To ensure the fulfillment of the hypothesis \eqref{eq:CC2pw}, we computed the minimum singular value of ${\rm T}$, then we rescaled ${\rm T}$ by multiplying it by $2$ divided by the minimum singular value multiplied 
by a random number in the interval $(0,1)$. 
To construct a sparse matrix ${\rm T}$ we used the Matlab routine {\it sprand}, which generates a sparse matrix with predefined dimension, density and singular values. 
First, we randomly generated the vector of singular values from a uniform distribution on $(0,1)$. 
The fulfillment of the hypothesis \eqref{eq:CC2pw} can be easily achieved by conveniently rescaling the singular values.
Finally, we evoke {\it sprand} with density equal to $0.004$. This means that, only about $0.4\%$ of the elements of ${\rm T}$ are non null.

\noindent {\bf (ii)} {\it Solution and vector $b$}: By Proposition~\ref{pr:uniqqpw}, equation \eqref{eq:pwls} has a unique solution if $\|{\rm T}^{-1}\|<1$. 
Note that if ${\rm x^*}=(x_1^*, {{\rm x_2^*}})  \in \mathbb{R}\times {\mathbb{R}}^{n-1}$ with  $x_1^* \leq -\|{\rm x_2^*}\|$ is a solution 
of \eqref{eq:pwls}, then ${\rm x^*}$ is a solution of ${\rm T}{\rm x} = b$. On the other hand, if $x_1^* \geq \|{\rm x_2^*}\|$, then ${\rm x^*}$
is a solution of $[\rm Id_n+{\rm T}]{\rm x} = b$. 
In both cases, the solution can be found by simply solving a linear system. 
In particular, if the generated sequence $\{\rm x^k\}$ converges to ${\rm x^*}=(x_1^*, {{\rm x_2^*}})  \in \mathbb{R}\times {\mathbb{R}}^{n-1}$ with
$x_1^* < -\|{\rm x_2^*}\|$ or $x_1^* > \|{\rm x_2^*}\|$ then the convergence is finite.
We ignore these trivial cases by assuming that the unique solution ${\rm x^*}$ of \eqref{eq:pwls} is such that $-\|{\rm x_2^*}\| < x_1^* <  \|{\rm x_2^*}\|$.
First, we randomly generated ${\rm x_2^*} \in {\mathbb{R}}^{n-1}$ from a uniform distribution on $(-10,10)$ and then we defined $x_1^*\in
\mathbb{R}$ as a convex combination between 
$-\|{\rm x_2^*}\|$ and $\|{\rm x_2^*}\|$.
After that, we computed  $b={\rm P}_{\cal{K}}({ {\rm x^*}}) +{\T} { {\rm x^*}}$.

\noindent {\bf (iii)} {\it Initial point}: As preliminary numerical tests, we investigated the influence of the starting point in the performance of the method. 
At this stage, we generated $100$ problems with fully dense $1000 \times 1000$ matrix ${\rm T}$ and $100$ problems with sparse $5000 \times 5000$ matrix ${\rm T}$. 
For each problem, we ran the semi-smooth Newton method starting from different initial points ${\rm x^0}=(x_1^0, {{\rm x_2^0}})  \in \mathbb{R}\times {\mathbb{R}}^{n-1}$ such that $x_1^0 > \|{\rm x_2^0}\|$, 
$x_1^0 < -\|{\rm x_2^0}\|$ and $-\|{\rm x_2^0}\| < x_1^0 <  \|{\rm x_2^0}\|$.
Observe that these regions are the interior of the cone $\cal{K}$, the interior of the polar cone ${\cal{K}}^{\circ}$ and the interior of the
complement of $\cal{K}\cup {\cal{K}}^{\circ}$, respectively.
For simplicity, let us call these regions by {\it Region 1, 2} and {\it 3}, respectively.
For dense instances, the method presented similar performance (in the sense of number of problems solved and average CPU time required) regardless
of the location of the initial point. 
%On the other hand, 
For sparse instances, the average CPU time to solve the problems was 8.61s, 8.14s, and 9.22s for the initial point in Region 1, 2, and 3, respectively.
Let us explain this slight difference.
When ${\rm T}$ is dense, the matrix of the linear system $[{\rm T}+{\rm V}(\rm x^0)]{\rm x}=b$ solved at the first iteration is also dense regardless
of the location of the initial point. 
On the other hand, for sparse ${\rm T}$, the matrix of the linear system $[{\rm T}+{\rm V}(\rm x^0)]{\rm x}=b$ is sparse if ${\rm x^0}$ belongs to Region 1 or 2, and dense if ${\rm x^0}$ belongs to Region 3, see Lemma~\ref{l:dpj}.
Since the method requires very few iterations to find the solution, the computational cost of the first iterations justify the difference between the average CPU times.
At this stage, we conclude that it is advantageous to take the starting point into Region 1 or 2. 
If ${\rm x^0}$ belongs to Region 2, then ${\rm V}({\rm x^0})=0$ and  the first iterate ${\rm x^1}$ is the unique solution of the linear system ${\rm T}{\rm x}=b$.
For simplicity, we assume directly that the starting point is given by the unique solution of ${\rm T}{\rm x}=b$.

For dense instances, we consider dimensions $n = 500$, $1000$, $2000$, and $3000$, and for sparse instances $n = 3000$, and $5000$.
We generate $200$ different problems for each test set.
In general, when ${\rm T}$ is a dense matrix, its condition number is of order $10^3$ or $10^4$.
For comparative purposes, the singular values of a sparse matrix ${\rm T}$ were rescaled so that its condition number is of order $10^4$.
Table \ref{tabletests} gives a summary of our numerical experiments. 
The column ``$n$'' is the dimension of the test set,  ``Cond(${\rm T}$)'' is the average condition number of matrices ${\rm T}$, ``Problems solved'' informs the number of successfully solved problems, 
``It'' and ``time'' are the average number of semi-smooth Newton iterations, and the average CPU time for the solved problems, respectively. 

\begin{table}[h]
{\footnotesize
\begin{center}
\begin{tabular}{|c|c|c|c|c|c|} \cline{2-6}
\multicolumn{1}{ c|  }{}      & $n$    & Cond(${\rm T}$)           & Problems solved  & It   & time (s) \\ \hline 
\multirow{4}{*}{Dense ${\rm T}$}    & 500    &  $1.27 \times 10^4$ & 198 (99.0\%)     & 1.97 & 0.03  \\ \cline{2-6} 
			      & 1000   &  $1.40 \times 10^4$ & 187 (93.5\%)     & 1.97 & 0.14 \\ \cline{2-6}
			      & 2000   &  $4.83 \times 10^4$ & 140 (70.0\%)     & 2.25 & 0.79 \\ \cline{2-6}
			      & 3000   &  $4.13 \times 10^4$ & 106 (53.0\%)     & 2.23 & 2.00 \\ \hline  \hline

\multirow{2}{*}{Sparse ${\rm T}$}   & 3000   &  $1.92 \times 10^4$ & 194 (97.0\%)     & 1.96 & 1.75  \\ \cline{2-6}
			      & 5000   &  $1.88 \times 10^4$ & 194 (97.0\%)     & 1.94 & 6.07  \\ \hline  
\end{tabular}
\caption{Performance of semi-smooth Newton method in sets of 200 random problems considering fully dense matrices $T$, and  sparse matrices $T$ (density approximately $0.4\%$).}
\label{tabletests}
\end{center}}
\end{table}

The semi-smooth Newton method solves a typical problem with two iterations. In fact, considering the 1019 solved problems for all instances, 
970 (95.2$\%$) problems were solved with 2 iterations, while 36 (3.5$\%$) problems were solved with 1 iteration and 13 (1.3$\%$) problems were solved with more than 2 iterations.
It is interesting to point out that, for any considered problem, {\it all} iterates ${\rm x}^k$ belongs to Region 3 (since the solution also belong to this set, this fact is not a big surprise). 
Therefore, the matrix ${\rm T}+{\rm V}(\rm x^k)$ is dense and the command {\it mldivide} of Matlab uses a LU solver for the associated linear system $[{\rm T}+{\rm V}(\rm x^k)]{\rm x}=b$.

The robustness of the semi-smooth Newton method is directly connected to the ability of the linear system solver used.
It is rarely true in practical implementations that direct methods for linear systems give the exact solution.
In some cases, they are not able to find the solution with high accuracy and the convergence of the main method gets impaired.
By means of numerical observations, we realize that in instances where the command {\it mldivide} is able to give the solution of a linear system with residuum less than $10^{-6}$, the semi-smooth Newton method 
stops reporting ``Solution found''. Otherwise, the semi-smooth Newton method is not able to find the solution with the desired accuracy.
Studies concerning the convergence theory of an inexact Newton method are necessary to clarify these issues.
As we can see in Table~\ref{tabletests}, in case of dense matrices ${\rm T}$, the number of problems solved by the semi-smooth Newton method decreases according to the increase of the dimension $n$. 
This phenomenon is clearly connected with the fact that the greater the dimension, the more operations are required to solve a linear equation. 
Consequently, the method is most affected by the accumulation of floating-point errors resulting in lower robustness. 
However, it is useful to mention that, in all cases of failure the achieved accuracy was close to the desired one (typically, order of $10^{-6}$). 
For sparse ${\rm T}$ instances, the command {\it mldivide} is able to solve linear systems with high accuracy and the robustness of the semi-smooth Newton method is not affected.

We close the computational results by testing the method in problems where ${\rm T}$ is a symmetric and positive definite matrix.
In this case, by Theorem~\ref{th:defpos}, equation~\eqref{eq:pwls} has a unique solution for any ${\rm b}\in\R^n$ and the semi-smooth Newton method \eqref{eq:newtonc2pw} is well-defined.
Moreover, convergence is guaranteed if $\|{\rm T}^{-1}\|<1$.
We randomly generated 200 problems with $1000 \times 1000$ matrices ${\rm T}$ that do {\it not} fulfill this hypothesis. Let us clarify this. 
First, we randomly generated a fully dense matrix ${\rm A}$  and a vector $\lambda\in\R^n$ of eigenvalues of ${\rm T}$ from a uniform distribution on $(0,1)$.
After that, we extracted the eigenvectors of matrix $({\rm A}+{\rm A}^\top)/2$ in the columns of a matrix ${\rm U}$ and defined ${\rm T}={\rm
U}{\rm D}{\rm U}^\top$, where ${\rm D}$ is the diagonal matrix 
with $(i, i)$-th entry equal to $\lambda_i$, $i = 1,\ldots, 1000$. Finally, we randomly generated the vector ${\rm b}$ as in the previous experiments. 
The average value of $\|{\rm T}^{-1}\|$ was $4.90 \times 10^3$.
The semi-smooth Newton method successfully solved {\it all} problems of this test set.
The average number of iterations and the average CPU time for solving the problems were 5.90 and 0.39 seconds, respectively.
We observed that, as in the previous tests, if we rescaled the eigenvalues of ${\rm T}$ such that hypothesis \eqref{eq:CC2pw} was fulfilled, the semi-smooth Newton method 
required (in general) two iterations for finding the solution with the desired accuracy.
When ${\rm T}$ is a symmetric and positive definite matrix, this experiment suggests the conjecture that the semi-smooth Newton method always converges. 

%%%%%%%%%%%%%%%%%%%%%%%%
\section{Final remarks} \label{sec:conclusions}
In this paper we studied a special  equation   associated to the second order cone. Our main  result shows that, under mild conditions, we can
apply  a semi-smooth Newton method for finding a  solution of  this equation and  the generated sequence converges globally and  linearly. The numerical experiments  suggest that the convergence rate is better than linear and even it achieves accurate
solutions of large scale problems in few iterations. The theoretical verification  of these properties remains as an open question.  Our numerical tests also  suggest  that the semi-smooth Newton method always converges if the matrix   of the equation is
symmetric and positive definite. The studied equation is important because it is related to linear second order cone complementarity problems, used
in a wide range of applications \cite{LoboVandenbergheBoyd98}. It would be interesting to see whether the used technique can be applied for solving  
a similar equation associated to  the cone of the semidefinite matrices. A more general open question is whether our semi-smooth Newton method approach can be
unified to solve linear symmetric cone complementarity problems. The importance of this question is due to its connections with physics,
mechanics, economics, game theory, robotics, optimization and neural networks, such as the ones described in
\cite{ZhangZhangPan2013,YonekuraKanno2012,AghassiBertsimas2006,NishimuraHayashiFukushima2012,LuoAnXia2009,NishimuraHayashFukushima2009,AndreaniFriedlanderMelloSantos2008,KoChenYang2011,ChenTseng2005}. We remark that any quadratic second order symmetric cone optimization problem can be reformulated in terms of linear complementarity problems 
related to symmetric cones, which is a further motivation to study this question.
We foresee further progress in this topic in the near future.   

%%%%%%%%%%%%%%%%%%%%%%%%%%%%%%%%%%%%%%%%%%%%%%%
%%%%%%%%%%%%%%%%%%%%%%%%%%%%%%%%%%%%%%%%%%%%%%%%%%%%

%\bibliographystyle{abbrv}
%\bibliography{NMPWLSLorentz}

\begin{thebibliography}{10}

\bibitem{AghassiBertsimas2006}
M.~Aghassi and D.~Bertsimas.
\newblock Robust game theory.
\newblock {\em Math. Program.}, 107(1-2, Ser. B):231--273, 2006.

\bibitem{AndreaniFriedlanderMelloSantos2008}
R.~Andreani, A.~Friedlander, M.~P. Mello, and S.~A. Santos.
\newblock Box-constrained minimization reformulations of complementarity
  problems in second-order cones.
\newblock {\em J. Global Optim.}, 40(4):505--527, 2008.


\bibitem{BBCFN2016}
J.~G. Barrios, J.~Y. Bello~Cruz, O.~P. Ferreira, and S.~Z. N{\'e}meth.
\newblock A semi-smooth {N}ewton method for a special piecewise linear system
  with application to positively constrained convex quadratic programming.
\newblock {\em J. Comput. Appl. Math.}, 301:91--100, 2016.

\bibitem{BFN2015}
J.~G. Barrios, O.~P. Ferreira, and S.~Z. N{\'e}meth.
\newblock Projection onto simplicial cones by {P}icard's method.
\newblock {\em Linear Algebra Appl.}, 480:27--43, 2015.

\bibitem{BFP2016}
J.~Y. Bello~Cruz, O.~P. Ferreira, and L.~F. Prudente.
\newblock On the global convergence of the inexact semi-smooth Newton method
  for absolute value equation.
\newblock {\em Comput. Optim. Appl.},  65(1):93--108, 2016.

\bibitem{BrugnanoCasulli2007}
L.~Brugnano and V.~Casulli.
\newblock Iterative solution of piecewise linear systems.
\newblock {\em SIAM J. Sci. Comput.}, 30(1):463--472, 2008.

\bibitem{BrugnanoCasulli2009}
L.~Brugnano and V.~Casulli.
\newblock Iterative solution of piecewise linear systems and applications to
  flows in porous media.
\newblock {\em SIAM J. Sci. Comput.}, 31(3):1858--1873, 2009.

\bibitem{CasulliZanolli2012}
V.~Casulli and P.~Zanolli.
\newblock Iterative solutions of mildly nonlinear systems.
\newblock {\em J. Comput. Appl. Math.}, 236(16):3937--3947, 2012.

\bibitem{ChenRavi2010}
J.~Chen and R.~P. Agarwal.
\newblock On {N}ewton-type approach for piecewise linear systems.
\newblock {\em Linear Algebra Appl.}, 433(7):1463--1471, 2010.

\bibitem{ChenTseng2005}
J.-S. Chen and P.~Tseng.
\newblock An unconstrained smooth minimization reformulation of the
  second-order cone complementarity problem.
\newblock {\em Math. Program.}, 104(2-3, Ser. B):293--327, 2005.

\bibitem{Clarke1990}
F.~H. Clarke.
\newblock {\em Optimization and nonsmooth analysis}, volume~5 of {\em Classics
  in Applied Mathematics}.
\newblock Society for Industrial and Applied Mathematics (SIAM), Philadelphia,
  PA, second edition, 1990.

\bibitem{FerreiraNemeth2014}
O.~Ferreira and S.~N\'{e}meth.
\newblock Projection onto simplicial cones by a semi-smooth Newton method.
\newblock {\em Optimization Letters}, pages 1--11, 2014.

\bibitem{FukushimaTseng2002}
M.~Fukushima, Z.-Q. Luo, and P.~Tseng.
\newblock Smoothing functions for second-order-cone complementarity problems.
\newblock {\em SIAM J. Optim.}, 12(2):436--460 (electronic), 2001/02.

\bibitem{GajardoSeeger14}
P.~Gajardo and A.~Seeger.
\newblock Equilibrium problems involving the {L}orentz cone.
\newblock {\em J. Global Optim.}, 58(2):321--340, 2014.

\bibitem{GriewankBerntRadonsStreubel2015}
A.~Griewank, J.-U. Bernt, M.~Radons, and T.~Streubel.
\newblock Solving piecewise linear systems in abs-normal form.
\newblock {\em Linear Algebra Appl.}, 471:500--530, 2015.

\bibitem{HiriartLemarecal1}
J.-B. Hiriart-Urruty and C.~Lemar{\'e}chal.
\newblock {\em Convex analysis and minimization algorithms: Fundamentals. {I}},
  volume 305 of {\em Grundlehren der Mathematischen Wissenschaften [Fundamental
  Principles of Mathematical Sciences]}.
\newblock Springer-Verlag, Berlin, 1993.

\bibitem{HuHuangQiong11}
S.-L. Hu, Z.-H. Huang, and Q.~Zhang.
\newblock A generalized {N}ewton method for absolute value equations associated
  with second order cones.
\newblock {\em J. Comput. Appl. Math.}, 235(5):1490--1501, 2011.

\bibitem{KanzowFerencziFukushima2009}
C.~Kanzow, I.~Ferenczi, and M.~Fukushima.
\newblock On the local convergence of semismooth {N}ewton methods for linear
  and nonlinear second-order cone programs without strict complementarity.
\newblock {\em SIAM J. Optim.}, 20(1):297--320, 2009.

\bibitem{KoChenYang2011}
C.-H. Ko, J.-S. Chen, and C.-Y. Yang.
\newblock Recurrent neural networks for solving second-order cone programs.
\newblock {\em Neurocomputing}, 74:3464--3653, 2011.

\bibitem{KongXiuHan08}
L.~Kong, N.~Xiu, and J.~Han.
\newblock The solution set structure of monotone linear complementarity
  problems over second-order cone.
\newblock {\em Oper. Res. Lett.}, 36(1):71--76, 2008.

\bibitem{LoboVandenbergheBoyd98}
M.~S. Lobo, L.~Vandenberghe, S.~Boyd, and H.~Lebret.
\newblock Applications of second-order cone programming.
\newblock {\em Linear Algebra Appl.}, 284(1-3):193--228, 1998.
\newblock ILAS Symposium on Fast Algorithms for Control, Signals and Image
  Processing (Winnipeg, MB, 1997).

\bibitem{LuoAnXia2009}
G.-M. Luo, X.~An, and J.-Y. Xia.
\newblock Robust optimization with applications to game theory.
\newblock {\em Appl. Anal.}, 88(8):1183--1195, 2009.

\bibitem{MalikMohan05}
M.~Malik and S.~R. Mohan.
\newblock On {$\bf Q$} and {${\bf R}_0$} properties of a quadratic
  representation in linear complementarity problems over the second-order cone.
\newblock {\em Linear Algebra Appl.}, 397:85--97, 2005.

\bibitem{Mangasarian2009}
O.~L. Mangasarian.
\newblock A generalized {N}ewton method for absolute value equations.
\newblock {\em Optim. Lett.}, 3(1):101--108, 2009.

\bibitem{MangasarianMeyer2006}
O.~L. Mangasarian and R.~R. Meyer.
\newblock Absolute value equations.
\newblock {\em Linear Algebra Appl.}, 419(2-3):359--367, 2006.

\bibitem{Moreau1962}
J.~J. Moreau.
\newblock D\'ecomposition orthogonale d'un espace hilbertien selon deux c\^ones
  mutuellement polaires.
\newblock {\em C. R. Acad. Sci.}, 255:238--240, 1962.

\bibitem{NemethZhang}
S.~N\'{e}meth and G.~Zhang. 
\newblock Conic optimization and complementarity problems.
\newblock {\em arXiv:1607.05161}, 2016.


\bibitem{NishimuraHayashFukushima2009}
R.~Nishimura, S.~Hayashi, and M.~Fukushima.
\newblock Robust {N}ash equilibria in {$N$}-person non-cooperative games:
  uniqueness and reformulation.
\newblock {\em Pac. J. Optim.}, 5(2):237--259, 2009.

\bibitem{NishimuraHayashiFukushima2012}
R.~Nishimura, S.~Hayashi, and M.~Fukushima.
\newblock Semidefinite complementarity reformulation for robust {N}ash
  equilibrium problems with {E}uclidean uncertainty sets.
\newblock {\em J. Global Optim.}, 53(1):107--120, 2012.

\bibitem{Ortega1990}
J.~M. Ortega.
\newblock {\em Numerical analysis}, volume~3 of {\em Classics in Applied
  Mathematics}.
\newblock Society for Industrial and Applied Mathematics (SIAM), Philadelphia,
  PA, second edition, 1990.
\newblock A second course.

\bibitem{LiSun93}
L.~Q. Qi and J.~Sun.
\newblock A nonsmooth version of {N}ewton's method.
\newblock {\em Math. Programming}, 58(3, Ser. A):353--367, 1993.

\bibitem{SunLeiLiu2015}
Z.~Sun, L.~Wu, and Z.~Liu.
\newblock A damped semismooth {N}ewton method for the {B}rugnano-{C}asulli
  piecewise linear system.
\newblock {\em BIT}, 55(2):569--589, 2015.

\bibitem{Wilkinson65}
J.~H. Wilkinson.
\newblock {\em The algebraic eigenvalue problem}.
\newblock Clarendon Press, Oxford, 1965.

\bibitem{YonekuraKanno2012}
K.~Yonekura and Y.~Kanno.
\newblock Second-order cone programming with warm start for elastoplastic
  analysis with von {M}ises yield criterion.
\newblock {\em Optim. Eng.}, 13(2):181--218, 2012.

\bibitem{ZhangZhangPan2013}
L.~L. Zhang, J.~Y. Li, H.~W. Zhang, and S.~H. Pan.
\newblock A second order cone complementarity approach for the numerical
  solution of elastoplasticity problems.
\newblock {\em Comput. Mech.}, 51(1):1--18, 2013.

\end{thebibliography}

\end{document}